\documentclass[pdflatex,sn-mathphys-ay]{sn-jnl}

\usepackage{graphicx}%
\usepackage{multirow}%
\usepackage{amssymb}
\usepackage{amsmath}
\usepackage{amsfonts}
\usepackage{mathrsfs}
\usepackage{mathtools}
\usepackage[title]{appendix}%
\usepackage{xcolor}%
\usepackage{textcomp}%
\usepackage{manyfoot}%
\usepackage{booktabs}%
\usepackage{algorithm}%
\usepackage{algorithmicx}%
\usepackage{algpseudocode}%
\usepackage{listings}%
\usepackage{enumitem}
\usepackage{url}
\usepackage{hyperref}

\usepackage{tikz}
\usetikzlibrary{arrows.meta,positioning,fit,calc,decorations.pathreplacing,shapes.multipart}
\usepackage{cleveref}

\theoremstyle{thmstyleone}
\newtheorem{theorem}{Theorem}[section]
\newtheorem{proposition}[theorem]{Proposition}
\newtheorem{lemma}[theorem]{Lemma}
\newtheorem{corollary}[theorem]{Corollary}
\theoremstyle{definition}
\newtheorem{definition}[theorem]{Definition}
\newtheorem{assumption}[theorem]{Assumption}
\theoremstyle{remark}
\newtheorem{remark}[theorem]{Remark}

\begin{document}

\title{From Microscopic Damage to Macroscopic Games: A Dimensionality Reduction of Stem Cell Homeostasis}

\author[1,2]{\fnm{Jiguang} \sur{Yu}}\email{jyu678@bu.edu}
\equalcont{These authors contributed equally to this work as co-first authors.}

\author*[3]{\fnm{Louis Shuo} \sur{Wang}}\email{wang.s41@northeastern.edu}
\equalcont{These authors contributed equally to this work as co-first authors.}

\author[4]{\fnm{Shihan} \sur{Ban}}\email{banshihan@uchicago.edu}
\equalcont{These authors contributed equally to this work as co-first authors.}

\affil[1]{\orgdiv{College of Engineering},
  \orgname{Boston University},
  \orgaddress{\city{Boston}, \postcode{02215}, \state{MA}, \country{USA}}}

\affil[2]{\orgdiv{GSS Fellow of Institute for Global Sustainability},
  \orgname{Boston University},
  \orgaddress{\city{Boston}, \postcode{02215}, \state{MA}, \country{USA}}}

\affil[3]{\orgdiv{Department of Mathematics},
  \orgname{Northeastern University},
  \orgaddress{\city{Boston}, \postcode{02115}, \state{MA}, \country{USA}}}

\affil[4]{\orgdiv{Committee on Computational and Applied Mathematics},
  \orgname{University of Chicago},
  \orgaddress{\city{Chicago}, \postcode{60637}, \state{IL},\country{USA}}}

\abstract{
Tissues must maintain macroscopic homeostasis despite the continuous microscopic accumulation of cellular damage. Theoretical models of this process often suffer from a disconnect between microscopic biophysics and macroscopic phenomenological games. Here, we bridge this gap by deriving an exact dimensionality reduction of a physiologically structured partial differential equation (PDE) into a low-dimensional dynamical system. Under the condition of uniform mortality, we mathematically demonstrate that tissue homeostasis operates as an induced Nash equilibrium, where the per-capita net growth rates of stem and differentiated phenotypes perfectly equalize. This reduction yields closed-form algebraic rules, the Ratio and Equalization Laws, that map continuous microscopic state dynamics to measurable macroscopic observables. To demonstrate the biological utility of this framework, we present a concrete, falsifiable case study of the murine intestinal crypt. By modeling crypt regeneration following irradiation-induced stem cell depletion, our framework successfully recovers the experimentally observed reliance on progenitor dedifferentiation. Furthermore, the model generates explicit, testable predictions, enabling the in vivo estimation of hard-to-measure lineage plasticity rates directly from aggregate static cell counts. This work provides a rigorous, predictive mathematical foundation for understanding how fast-renewing tissues filter microscopic noise to sustain macroscopic regenerative capacity.}


\maketitle

\section*{Author summary}
How do our bodies maintain healthy organs when individual cells are constantly accumulating damage and dying? To survive, tissues rely on stem cells to replace lost tissue, and sometimes even allow mature cells to revert back into stem cells to help with emergency repairs. However, understanding exactly how tissues coordinate this complex balancing act has been historically difficult. Mathematical models usually focus either on the microscopic physical wear-and-tear of single cells or on the macroscopic "game" played by the whole cell population, but rarely both.

In this study, we bridge that gap. We developed a mathematical method that directly translates the complex physical rules of cell damage into a much simpler game-theory framework. We discovered that a healthy, stable tissue acts like a perfectly balanced game—a "Nash equilibrium"—where neither stem cells nor mature cells hold a competitive growth advantage over the other. By applying our theory to the rapid regeneration of the mouse intestine, we demonstrate that our equations can accurately predict how tissues heal after severe injury. Ultimately, our work provides biologists with a practical new tool: the ability to calculate hidden, complex tissue healing rates using only simple snapshot measurements of cell populations.

\section{Introduction}

Multicellular tissues face a fundamental conflict:
they must preserve macroscopic stability (homeostasis) while individual cells continually accumulate molecular damage and drift toward disorder 
\citep{xavier_da_silveira_dos_santos_single_2019,guillot_mechanics_2013}.
In rapidly renewing systems like the intestinal epithelium and hematopoietic system, stem and progenitor cells are continuously exposed to endogenous and environmental stressors  \citep{chiacchiera_transcriptional_2019,beumer_regulation_2016}.
Over time, repeated challenges produce a broad cell-to-cell distribution of defects---ranging from DNA lesions to metabolic decline---that acts as a dynamic `hidden coordinate' of damage.
As this coordinate evolves, it reshapes intracellular signaling, biases fate choices, and undermines long-term tissue maintenance \citep{jones_stem_2023,mcneely_dna_2020,chatterjee_aging_2021,mi_mechanism_2022,liu_stem_2022,oh_stem_2014}.
More generally, reciprocal feedback in multiscale biological systems can generate not only homeostatic regulation but also explicit transitions to spatial niche formation and patterned heterogeneity in spatially extended settings \citep{yu_chemotactic_2026}.

To cope with this heterogeneity, organisms exploit local regulatory feedback and cellular plasticity \citep{lou_homeostasis_2019,biteau_maintaining_2011,bates_agent-based_2023}.
Differentiated cells can revert to a proliferative stage through dedifferentiation to replace lost tissue \citep{jopling_dedifferentiation_2011, blanpain_plasticity_2014}.
For example, following the ablation of native stem cells in the intestine or airway, committed progenitors can actively regain stem-like properties to repair the epithelium \citep{murata_ascl2-dependent_2020, tata_dedifferentiation_2013, tetteh_replacement_2016}. 
Yet, this plasticity is a double-edged sword: while supporting repair, its misregulation can promote dysplasia or malignant transformation \citep{van_der_heijden_stem_2019,beumer_regulation_2016,chiacchiera_transcriptional_2019}, functioning effectively as a high-stakes strategy in the tissue's struggle for survival \citep{bensellam_mechanisms_2018, nordmann_role_2017}.

\subsection{The challenge of scale separation: Physics vs. Games}
Understanding how tissues coordinate this plasticity requires bridging two theoretical scales.
At the microscopic scale, physiologically structured partial differential equations (PDEs) capture continuous damage accumulation and division events 
\citep{magal_theory_2018,perthame_transport_2007,cui_hopf-bifurcation_2024,zhang_periodic_2019}.
However, structured PDEs are mechanistic but hard to interpret; their infinite-dimensional nature obscures simple, testable laws for the whole tissue.

At the macroscopic level, evolutionary game theory elegantly models ecological mechanisms of tissue coordination \citep{zhu_integrating_2016,huang_stochastic_2015,argasinski_interaction_2018,helbing_stochastic_1996}. Yet, a critical gap remains: game-theoretic ``payoff functions'' are typically phenomenological, assuming timescale separations that often fail in regenerating tissues \citep{dieckmann_dynamical_1996,dercole_analysis_2008,metz_adaptive_1995,geritz_evolutionarily_1998}. Without a rigorous derivation linking continuous micro-physics directly to macro-game competitive advantages, it remains unclear whether game-theoretic stability is a genuine biological property or a mathematical artifact.
Recent work has connected stem--TD lineage dynamics across scales, from density-dependent continuoustime Markov chain models to diffusion approximations and aggregate total-mass laws, including ratio and equalization constraints in related damage-structured settings \citep{yu_dedifferentiation_2026}. The present work differs in deriving an exact deterministic closure and an induced game-theoretic interpretation directly from the transport model.

\subsection{Contributions}
In this work, we bridge this scale gap by deriving an exact closure of the total-mass balance laws into a planar ODE under uniform TD mortality. We move beyond metaphor to demonstrate that the ``game'' played by cells is an emergent physical property of lineage dynamics (Figure~\ref{fig:multiscale_pipeline}). Our specific contributions offer new theoretical and practical insights:
\begin{enumerate}[label=(\roman*)]
\item \textbf{Exact Dimensionality Reduction:} We identify a structural condition under which infinite-dimensional PDE dynamics collapse exactly into a two-variable system. We emphasize that this condition—uniform mortality among terminally differentiated (TD) cells—serves as a mathematical closure condition rather than a universal biological claim. It provides a rigorous justification for using simple compartment models to study complex, heterogeneous tissues.
\item \textbf{Homeostasis as a Nash Equilibrium:} We establish that tissue homeostasis is mathematically equivalent to a Nash equilibrium in an induced evolutionary game. By deriving payoff functions'' directly from cell-cycle parameters, we show stability is achieved when stem and TD phenotypes reach fitness equalization'' (equal per-capita net growth rates).
\item \textbf{A Quantitative Extinction Threshold:} We establish a sharp ``extinction-growth'' threshold by analyzing the global phase-plane structure, proving that under compensatory feedback, the system globally converges to this homeostatic Nash equilibrium.
\item \textbf{Biological Validation (Murine Intestinal Crypts):} We apply these exact laws to a concrete, falsifiable case study of the murine intestinal crypt. 
The framework successfully models dedifferentiation-driven regeneration and provides a practical tool for biologists, allowing hard-to-measure rates (such as dedifferentiation flux) to be estimated directly from static, macroscopic flow cytometry observables.
\end{enumerate}

\begin{figure*}[htbp]
    \centering
    \includegraphics[width=\textwidth]{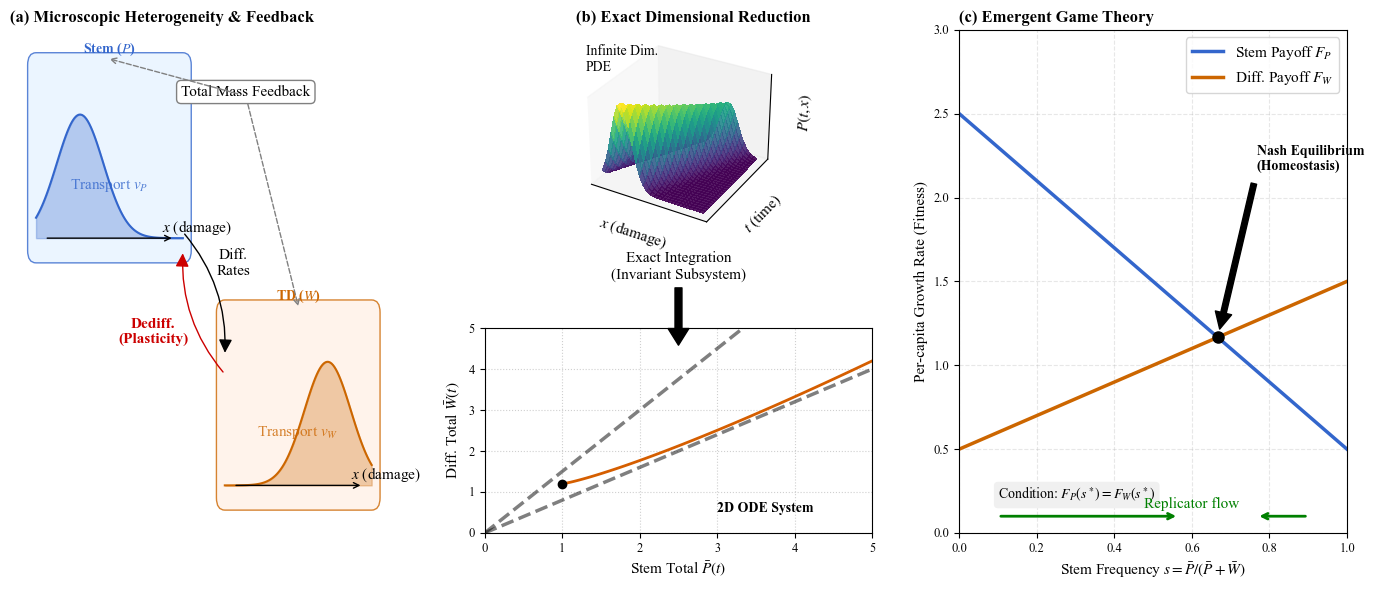} 
    \caption{\textbf{From microscopic heterogeneity to macroscopic game theory.} 
    \textbf{(a) Microscopic scale:} Stem ($P$, blue) and TD ($W$, orange) cells are heterogeneous populations distributed over a damage coordinate $x$.
    \textbf{(b) Exact reduction:} The infinite-dimensional PDE dynamics (top surface) collapse exactly onto a 2D subsystem for totals (bottom phase plane) under uniform mortality.
    \textbf{(c) Macroscopic law:} Homeostasis emerges as a Nash equilibrium where per-capita growth rates (payoffs) equalize. Green arrows indicate the replicator flow.}
    \label{fig:multiscale_pipeline}
\end{figure*}

\subsection{Most direct experimental settings and falsifiable readouts}
The exact reduction and induced game structure are most directly testable in fast-renewing tissues where progenitor output is feedback-regulated and injury elicits measurable plasticity. Two canonical settings are rapidly renewing epithelia (e.g., intestinal crypts, airway epithelium) \citep{gall_homeostasis_2023} and the hematopoietic system \citep{zhao_regulation_2015,safina_new_2024}.

In these paradigms, our theory makes three falsifiable predictions at the level of aggregate compartment readouts, testable via standard lineage tracing, pulse-chase, or EdU-BrdU proliferation assays. The framework is refuted if measured totals violate the closed balance laws under uniform mortality (Corollary~\ref{cor:uniform-death-closure}; Fig.~\ref{fig:closure_condition}):
\begin{enumerate}[label=(\roman*)]
\item \textit{Ratio Law and Equalization Law:} At steady state, the stem-to-TD ratio must satisfy $\bar P^*/\bar W^*=\delta/\lambda_P(\bar W^*)$. Similarly, the required dedifferentiation rate must match $\lambda_R(\bar P^*)=\delta\big(p_2(\bar W^*)-p_1(\bar W^*)\big)$. Both quantities can be estimated from aggregate counts combined with division, death, and return-flux assays.
\item \textit{Nash Equilibrium:} The Nash/homeostasis condition is falsifiable as payoff equalization in observables. At the homeostatic point, both per-capita net growth rates vanish and coincide: $F_P(\bar P^*,\bar W^*)=F_W(\bar P^*,\bar W^*)=0$.
\item \textit{Extinction--Growth Threshold}: This threshold is falsifiable through open-loop perturbations. Varying the renewal bias $\Delta p=\hat p_1-\hat p_2$ or the TD death rate $\delta$ should switch the origin from stable (extinction) to saddle (growth) precisely at $\Delta p_{\mathrm{crit}}=-\hat\lambda_R/\delta$.
\end{enumerate} 

The structure of the work goes as follows. Section~\ref{sec:model} formulates the damage-structured PDE and establishes well-posedness. Section~\ref{sec:balance_reduction} derives exact balance laws and proves that under uniform mortality the infinite-dimensional dynamics close exactly to a two-dimensional ordinary differential equation (ODE). Section~\ref{sec:replicator_game} identifies the induced replicator structure and shows homeostasis is a Nash equilibrium with closed-form game-theoretic laws. Section~\ref{sec:dynamics} establishes the extinction-growth threshold, proves global convergence via Bendixson--Dulac analysis, and rules out periodic orbits. Section~\ref{sec:numerics} provides computational verification of all main results. Section~\ref{sec:case_study} presents the case study on murine intestinal crypt homeostasis and regeneration. Section~\ref{sec:discussion} discusses the experimental interpretation of our results and provides limitations and extensions of the present study.

\section{Damage-structured PDE model}
\label{sec:model}

\subsection{State variables and processes}
\label{subsec:state_processes}
We consider a two-compartment stem--TD lineage in which cells accumulate a scalar damage coordinate $x\ge 0$. Let
$P(t,x)\ge 0$ and $W(t,x)\ge 0$ denote, respectively, the stem and TD densities at time $t\ge 0$ and damage level $x\ge 0$. The corresponding total masses are
\begin{equation}\label{eq:totals_def}
\bar P(t)=\int_0^\infty P(t,x)\,dx,\qquad
\bar W(t)=\int_0^\infty W(t,x)\,dx.
\end{equation}
Damage accumulates deterministically along characteristics at constant speeds $v_P,v_W>0$, yielding transport terms $v_P\partial_x P$ and $v_W\partial_x W$, as in classical structured transport models \citep{perthame_transport_2007,perthame_introduction_2023}. Stem cells divide with a feedback-regulated rate $\lambda_P(\bar W)$, and upon division each daughter inherits a fixed fraction of the mother's damage (deterministic inheritance). Concretely, a division event produces:
(i) two stem daughters with probability $p_1(\bar W)$ and damage fractions $\alpha_1,\alpha_2\in(0,1)$;
(ii) two TD daughters with probability $p_2(\bar W)$ and damage fractions $\beta_1,\beta_2\in(0,1)$;
(iii) one stem and one TD daughter with probability $p_3(\bar W):=1-p_1(\bar W)-p_2(\bar W)$ and damage fractions $\gamma_1,\gamma_2\in(0,1)$.
Conservation at division imposes
\begin{equation}\label{eq:inherit_cons}
\alpha_1+\alpha_2=\beta_1+\beta_2=\gamma_1+\gamma_2=1.
\end{equation}

The resulting nonlocal birth operators take the standard ``push-forward'' form induced by deterministic inheritance:
\begin{align}
\label{eq:birth_ops}
\mathcal B_P[P](t,x) &=
\frac{p_1(\bar W(t))\lambda_P(\bar W(t))}{\alpha_1}P\!\left(t,\frac{x}{\alpha_1}\right)
+\frac{p_1(\bar W(t))\lambda_P(\bar W(t))}{\alpha_2}P\!\left(t,\frac{x}{\alpha_2}\right) \nonumber\\[6pt]
&+\frac{p_3(\bar W(t))\lambda_P(\bar W(t))}{\gamma_1}P\!\left(t,\frac{x}{\gamma_1}\right),\\[6pt]
\nonumber
\mathcal B_W[P](t,x) &=
\frac{p_2(\bar W(t))\lambda_P(\bar W(t))}{\beta_1}P\!\left(t,\frac{x}{\beta_1}\right)
+\frac{p_2(\bar W(t))\lambda_P(\bar W(t))}{\beta_2}P\!\left(t,\frac{x}{\beta_2}\right) \\[6pt]
&+\frac{p_3(\bar W(t))\lambda_P(\bar W(t))}{\gamma_2}P\!\left(t,\frac{x}{\gamma_2}\right).
\end{align}
The concrete feedback functions $\lambda_P$, $\lambda_R$, and $p_i$ will appear in Section~\ref{subsec:feedback_assumptions}.
In addition, TD cells can revert to the stem compartment via dedifferentiation $W\to P$ with rate $\lambda_R(\bar P)$, and TD cells die at a damage-dependent rate $\delta(x)\ge 0$ (Section~\ref{subsec:feedback_assumptions}). Putting these components together yields the transport--reaction system \citep{wang_damage-structured_2026}
\begin{equation}\label{eq:PDE_model}
\begin{cases}
\partial_t P + v_P \partial_x P
= \mathcal B_P[P](t,x) - \lambda_P(\bar W(t))\,P(t,x) + \lambda_R(\bar P(t))\,W(t,x),\\[4pt]
\partial_t W + v_W \partial_x W
= \mathcal B_W[P](t,x) - \big(\delta(x)+\lambda_R(\bar P(t))\big)\,W(t,x),
\end{cases} \qquad x>0,\ t>0.
\end{equation}
Since $v_P,v_W>0$, characteristics enter the domain at $x=0$. We impose homogeneous inflow boundary conditions
\begin{equation}\label{eq:inflow_bc}
P(t,0)=W(t,0)=0,\qquad t\ge 0,
\end{equation}
so that newborn mass at small damage is generated by the nonlocal birth terms rather than injected through the boundary; this is consistent with standard structured-transport formulations and ensures boundary terms vanish in total-mass integrations \citep{perthame_introduction_2023,liang_global_2025}. Initial data are prescribed by
\begin{equation}\label{eq:IC}
P(0,x)=P_0(x),\qquad W(0,x)=W_0(x),\qquad x>0.
\end{equation}

\subsection{Feedback maps and modeling assumptions}
\label{subsec:feedback_assumptions}
Numerous biological experiments have confirmed the existence of negative feedback regulators, where TD cells secrete signaling molecules to inhibit stem cell self-renewal and proliferation. For instance, within the TGF-$\beta$ superfamily, GDF11 in the olfactory epithelium, GDF-8 (myostatin) in skeletal muscle, and BMP3 in bone act as critical molecular brakes that prevent hyperplasia by restricting progenitor expansion and tissue mass \citep{wu_autoregulation_2003, mcpherron_regulation_1997, daluiski_bone_2001}. Furthermore, similar feedback loops are essential for stabilizing rapidly renewing tissues, such as the Cxcl12-mediated maintenance of hematopoietic stem cell quiescence and the autocrine TGF-$\beta$ mechanisms that limit keratinocyte overgrowth \citep{tzeng_loss_2011, yamasaki_keratinocyte_2003}.

At the same time, we consider the suppression of dedifferentiation by direct contact of TD cells with stem cells \citep{tata_dedifferentiation_2013,guo_dedifferentiation_2022}. 
Therefore, regulation is encoded through monotone (typically decreasing) Hill-type feedback maps,
\begin{align}\label{eq:Hill_maps}
\begin{aligned}
p_1(\bar W)=\frac{\hat p_1}{1+(k_1\bar W)^{m_1}},\quad
p_2(\bar W)=\frac{\hat p_2}{1+(k_2\bar W)^{m_2}},\\[4pt]
\lambda_P(\bar W)=\frac{\hat\lambda_P}{1+(k_3\bar W)^{m_3}},\quad
\lambda_R(\bar P)=\frac{\hat\lambda_R}{1+(k_4\bar P)^{m_4}},
\end{aligned}
\end{align}
with gains $k_i>0$, Hill exponents $m_i\ge 1$, and open-loop (baseline) values
$\hat p_i\coloneqq p_i(0)$, $\hat\lambda_P\coloneqq \lambda_P(0)$, $\hat\lambda_R\coloneqq \lambda_R(0)$.

We require $p_1,p_2\in[0,1]$ and $p_1(\bar W)+p_2(\bar W)\le 1$ for all $\bar W\ge 0$
so that $p_3(\bar W)\coloneqq 1-p_1(\bar W)-p_2(\bar W)\in[0,1]$.

The TD death rate $\delta(x)$ is assumed bounded and nonnegative. We distinguish two regimes:
\begin{enumerate}[label=(\roman*), leftmargin=2.2em]
\item \textbf{General bounded death $\delta(x)$:} used to establish global well-posedness and exact balance laws for totals (Theorem~\ref{thm:balance-laws}).
\item \textbf{Uniform death $\delta(x)\equiv\delta$:} 
we consider the specific regime where environmental attrition dominates intrinsic mortality. 
\end{enumerate}

Uniform TD mortality is used here as a structural identifiability condition for exact closure, not as a claim that real tissues have perfectly damage-independent death.
Exact closure on the two totals is guaranteed when $\delta(x)\equiv\delta$, because then the mortality flux reduces to $\delta \bar W(t)$. More generally, any regime in which the mortality flux depends only on $\bar W(t)$ would also yield closure, but constant mortality is the canonical structurally identifiable case considered here.
Therefore, the uniform mortality allows us to establish a baseline `neutral' model against which age-dependent mortality can be compared.
In that sense, “uniform death” is identifiable at the level of aggregate time series: failure of the closed balance laws is itself a diagnostic of non-uniform mortality.
Figure~\ref{fig:mechanisms} illustrates damage accumulation and inheritance mechanics and shows the feedback regulation for self-renewal and differentiation probabilities. 

\begin{figure*}[htbp]
 \centering
 \includegraphics[width=\textwidth]{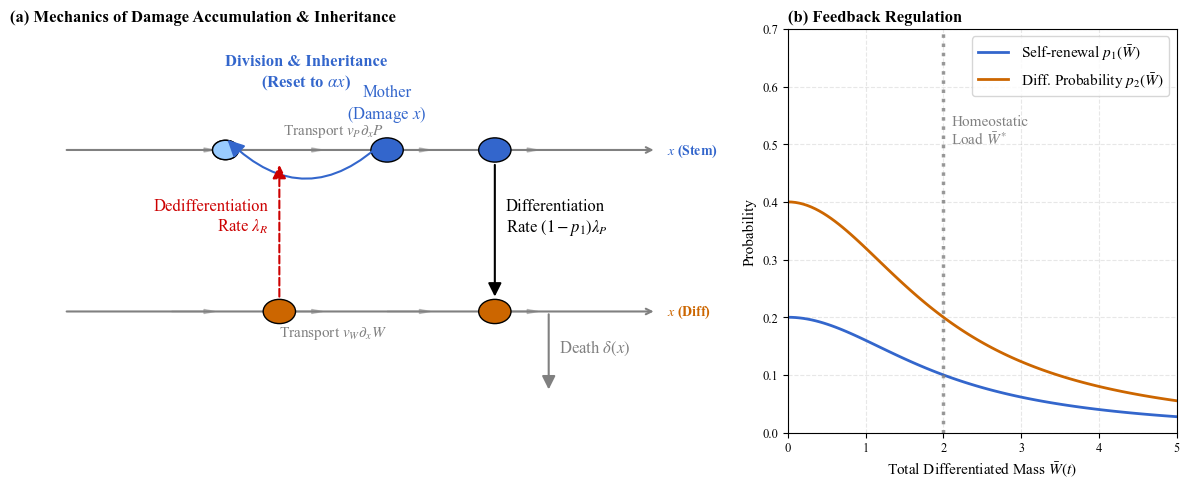} 
 \caption{\textbf{Mechanisms of damage dynamics and regulation.} 
 \textbf{(a) Conveyor-belt dynamics:} Damage accumulates via transport (gray arrows). Division (blue arc) resets damage state ($\alpha x < x$), while dedifferentiation (red) couples the compartments.
 \textbf{(b) Feedback logic:} Hill-type regulation of self-renewal ($p_1$) and differentiation ($p_2$) probabilities stabilizes the tissue burden $\bar{W}$ at the homeostatic load $\bar{W}^*$.}
 \label{fig:mechanisms}
\end{figure*}

For the analysis below we adopt the following standing assumptions.
\begin{assumption}\label{ass:basic}
\begin{enumerate}[label=(H\arabic*), leftmargin=2.6em]
\item \label{H1}
The maps $p_1,p_2,\lambda_P,\lambda_R$ are bounded and globally Lipschitz on $[0,\infty)$
(e.g.\ the Hill families in \eqref{eq:Hill_maps} satisfy this).
\item \label{H2}
$\delta\in L^\infty([0,\infty))$ is measurable and nonnegative.
\item \label{H3}
$P_0,W_0\in L^1_+([0,\infty))$.
\end{enumerate}
\end{assumption}

\subsection{Well-posedness setting and weak/mild solutions}
\label{subsec:wellposedness_setting}
Fix $T>0$ and define the state space
\begin{equation}\label{eq:XT_def}
\mathcal X_T \coloneqq C\!\big([0,T];L^1([0,\infty))\times L^1([0,\infty))\big),
\qquad
\|(P,W)\|_{\mathcal X_T} \coloneqq \sup_{t\in[0,T]}\big(\|P(t,\cdot)\|_1+\|W(t,\cdot)\|_1\big),
\end{equation}
with $\|\cdot\|_1=\|\cdot\|_{L^1([0,\infty))}$. Under Assumption~\ref{ass:basic},
system \eqref{eq:PDE_model} can be formulated in mild form along characteristics. Let
\[
X_P(s;t,x)=x-v_P(t-s),\qquad
X_W(s;t,x)=x-v_W(t-s),\qquad 0\le s\le t,
\]
and define integrating factors
\[
\Lambda_P(s;t)=\int_s^t \lambda_P(\bar W(\tau))\,d\tau,\qquad
\Lambda_W(s;t,x)=\int_s^t\big(\delta(X_W(\tau;t,x))+\lambda_R(\bar P(\tau))\big)\,d\tau.
\]
A mild solution $(P,W)\in\mathcal X_T$ is then a fixed point of the Picard map obtained by inserting
\eqref{eq:birth_ops} into the characteristic representations (cf.\ standard semigroup/characteristic
constructions for transport equations with inflow boundary conditions \citep{perthame_introduction_2023}).
The global Lipschitz property in \ref{H1} yields a contraction for small $T$; positivity is preserved by the nonnegativity of kernels and rates; and an a priori $L^1$-mass bound provides continuation to all times.
We therefore obtain a unique global nonnegative solution in
$\mathcal X_\infty := C([0,\infty);L^1\times L^1)$, together with exact balance laws for
$\bar P$ and $\bar W$ (Theorem~\ref{thm:balance-laws}).

\section{Balance laws and exact reduction}
\label{sec:balance_reduction}

\subsection{Balance laws for totals (general \texorpdfstring{$\delta(x)$}{delta(x)})}
\label{subsec:balance_general_delta}
We first derive exact balance laws for the total masses $\bar P(t)$ and $\bar W(t)$ defined in
\eqref{eq:totals_def}, under a general bounded, nonnegative death profile $\delta(x)$.

\begin{theorem}[Balance laws for totals]
\label{thm:balance-laws}
Assume Assumption~\ref{ass:basic}. Then the PDE system \eqref{eq:PDE_model}--\eqref{eq:IC}
admits a unique global nonnegative solution
$(P,W)\in C([0,\infty);L^1([0,\infty))\times L^1([0,\infty)))$.
Moreover, the total masses satisfy, for a.e.\ $t>0$,
\begin{equation}\label{eq:balance_laws_general}
\begin{cases}
\displaystyle
\frac{d\bar P}{dt}
=\Big(p_1(\bar W)-p_2(\bar W)\Big)\lambda_P(\bar W)\,\bar P
+\lambda_R(\bar P)\,\bar W,\\[8pt]
\displaystyle
\frac{d\bar W}{dt}
=\Big(1-p_1(\bar W)+p_2(\bar W)\Big)\lambda_P(\bar W)\,\bar P
-\lambda_R(\bar P)\,\bar W
-\int_0^\infty \delta(x)\,W(t,x)\,dx,
\end{cases}
\end{equation}
with initial conditions $\bar P(0)=\|P_0\|_{1}$ and $\bar W(0)=\|W_0\|_{1}$.
\end{theorem}

\begin{proof}[Proof sketch (derivation of \eqref{eq:balance_laws_general})]
We outline the formal computation; the rigorous argument follows by testing the weak formulation
against cutoffs $\psi_R(x)$ and passing $R\to\infty$, where $\psi_R \equiv 1$ on $[0,R]$, $\operatorname{supp} \psi_R \subset [0,R+1]$ and $\lVert\psi_R' \rVert_{\infty}\le C$ with $C$ independent of $R$ (cf.\ standard transport weak formulations).

Integrate the $P$-equation in \eqref{eq:PDE_model} over $x\in(0,\infty)$:
\[
\frac{d}{dt}\int_0^\infty P\,dx
+\int_0^\infty v_P\partial_x P\,dx
=\int_0^\infty \mathcal B_P[P]\,dx
-\lambda_P(\bar W)\int_0^\infty P\,dx
+\lambda_R(\bar P)\int_0^\infty W\,dx.
\]
The transport term gives the boundary flux
$\int_0^\infty v_P\partial_x P\,dx = v_P P(t,\infty)-v_P P(t,0)$.
Under integrability and the inflow boundary condition \eqref{eq:inflow_bc}, this term vanishes:
$P(t,\infty)=0$ (in the sense of $L^1$ tails) and $P(t,0)=0$ (in the sense of traces for $L^1$ solutions to transport equations).
An identical computation holds for the $W$-equation.

It remains to compute $\int \mathcal B_P[P]\,dx$ and $\int \mathcal B_W[P]\,dx$.
Using the change of variables $y=x/\alpha$,
\[
\int_0^\infty \frac{1}{\alpha}P\!\left(t,\frac{x}{\alpha}\right)\,dx
=\int_0^\infty P(t,y)\,dy=\bar P(t),
\]
and similarly for $\alpha\in\{\alpha_i,\beta_i,\gamma_i\}$, together with $p_1+p_2+p_3=1$,
one obtains
\[
\int_0^\infty \mathcal B_P[P]\,dx=(2p_1(\bar W)+p_3(\bar W))\lambda_P(\bar W)\bar P
=\big(1+p_1(\bar W)-p_2(\bar W)\big)\lambda_P(\bar W)\bar P,
\]
\[
\int_0^\infty \mathcal B_W[P]\,dx=(2p_2(\bar W)+p_3(\bar W))\lambda_P(\bar W)\bar P
=\big(1-p_1(\bar W)+p_2(\bar W)\big)\lambda_P(\bar W)\bar P.
\]
Substituting these identities yields \eqref{eq:balance_laws_general}.
\end{proof}

\paragraph{Why the inflow boundary removes flux terms}
Because $v_P,v_W>0$, characteristics enter the half-line through $x=0$ (an inflow boundary).
The condition \eqref{eq:inflow_bc} enforces zero boundary trace at $x=0$,
so the integrated transport flux produces no boundary contribution in the balance laws
(after justification via cutoff testing and $R\to\infty$).

\subsection{Exact two-dimensional closure under uniform death}
\label{subsec:closure_uniform_delta}
The balance laws \eqref{eq:balance_laws_general} are not closed in general, because
the mortality term depends on the \emph{distribution} $W(t,\cdot)$ through
$\int_0^\infty \delta(x)W(t,x)\,dx$. A key structural simplification occurs when the
TD death rate is constant.

\begin{corollary}[Exact two-compartment closure for $\delta(x)\equiv\delta$]
\label{cor:uniform-death-closure}
If $\delta(x)\equiv\delta>0$ is constant, then
$\int_0^\infty \delta(x)W(t,x)\,dx=\delta\,\bar W(t)$ and the balance laws close exactly to the planar ODE system
\begin{align}\label{eq:reduced_balance}
\begin{aligned}
&\begin{cases}
\displaystyle
\dot{\bar P}
=\Big(p_1(\bar W)-p_2(\bar W)\Big)\lambda_P(\bar W)\,\bar P
+\lambda_R(\bar P)\,\bar W,\\[8pt]
\displaystyle
\dot{\bar W}
=\Big(1-p_1(\bar W)+p_2(\bar W)\Big)\lambda_P(\bar W)\,\bar P
-\big(\delta+\lambda_R(\bar P)\big)\,\bar W,
\end{cases}
 \\[6pt] 
 &(\bar P(0),\bar W(0))=(\|P_0\|_1,\|W_0\|_1).
 \end{aligned}
\end{align}
The closed quadrant $\{\bar P\ge 0,\ \bar W\ge 0\}$ is forward invariant, and solutions are global.
\end{corollary}

\begin{proof}[Proof sketch]
Closure is immediate from $\int_0^\infty \delta W\,dx=\delta\bar W$.
Forward invariance follows from evaluating the vector field on the boundary:
if $\bar P=0$ and $\bar W\ge 0$, then $\dot{\bar P}=\lambda_R(0)\bar W\ge 0$;
if $\bar W=0$ and $\bar P\ge 0$, then $\dot{\bar W}=(1-p_1(0)+p_2(0))\lambda_P(0)\bar P\ge 0$.
For global existence, set $M(t):=\bar P(t)+\bar W(t)$ and compute from \eqref{eq:reduced_balance}
\begin{equation}\label{eq:M_balance}
\dot M(t)=\lambda_P(\bar W(t))\,\bar P(t)-\delta\,\bar W(t)
\le \hat\lambda_P\,M(t),
\end{equation}
where $\hat\lambda_P=\lambda_P(0)$ and we used $\lambda_P(\cdot)\le\hat\lambda_P$ and $\bar P\le M$.
Gronwall's inequality yields $M(t)\le M(0)e^{\hat\lambda_P t}$ for all $t\ge 0$,
precluding finite-time blow-up and implying global existence.
\end{proof}

\subsection{What breaks when \texorpdfstring{$\delta(x)$}{delta(x)} is non-constant}
\label{subsec:nonconstant_delta_breaks}
When $\delta(x)$ varies with damage, the mortality term in \eqref{eq:balance_laws_general} becomes
\[
\int_0^\infty \delta(x)\,W(t,x)\,dx,
\]
which depends on the full distribution $W(t,\cdot)$ rather than the single scalar $\bar W(t)$.
Consequently, the pair $(\bar P,\bar W)$ does not form a closed dynamical system in general:
additional moments (e.g.\ $\int \delta(x)W\,dx$, or other weighted integrals) are required, and
any finite-dimensional closure would typically be approximate. 
Such approximate closures are common in hybrid and stochastic multiscale models, where deterministic mean-field limits are obtained by closing higher-order correlations (e.g.\ first-moment closure) to yield tractable PDE surrogates \citep{wang_analysis_2025}.
Figure~\ref{fig:closure_condition} compares the uniform death rate and variable death rate scenarios. It provides the mechanistic contrast: when $\delta(x)\equiv\delta$, mortality depends only on total mass and the PDE totals obey the closed ODE exactly; when $\delta(x)$ varies with damage, the mortality flux depends on the shape of $W(t,\cdot)$, breaking closure and requiring additional moments.

\begin{figure*}[htbp]
 \centering
 \includegraphics[width=\textwidth]{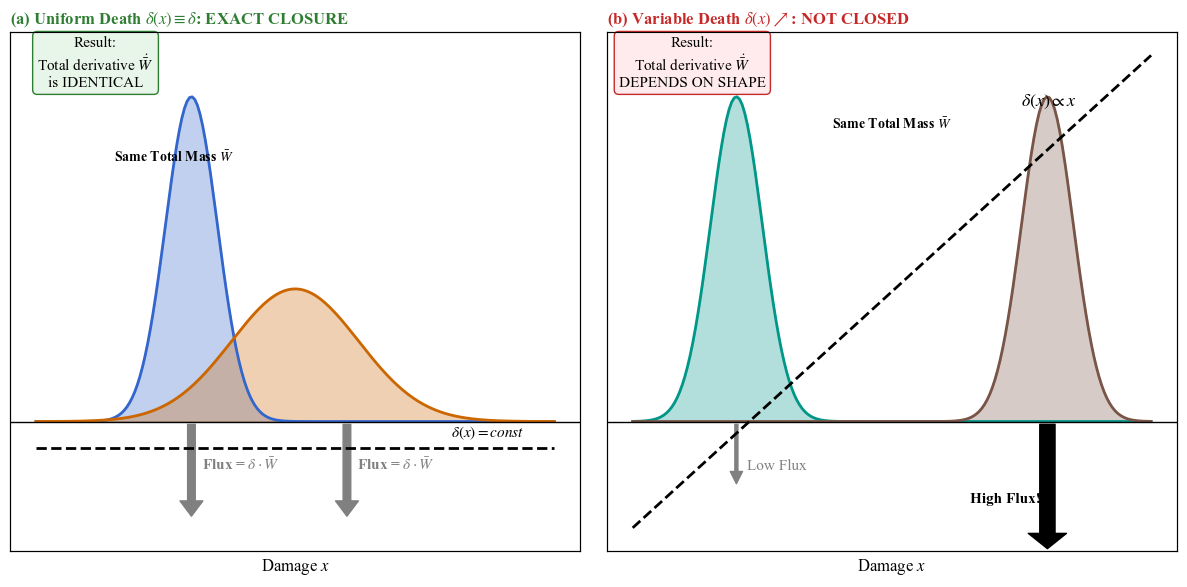} 
 \caption{\textbf{Why uniform mortality enables exact dimensional reduction.}
 \textbf{(a) Uniform death ($\delta(x) \equiv \delta$):} Mortality flux depends only on total mass (area), ensuring exact closure regardless of distribution shape.
 \textbf{(b) Variable death ($\delta(x) \nearrow$):} Mortality flux depends on the distribution's shape (e.g., age structure), breaking the ODE closure. 
 This figure also serves as the mechanistic counterfactual: any observed deviation of total-mass trajectories from \eqref{eq:reduced_balance} (beyond numerical error) indicates a non-constant effective mortality and hence failure of exact closure.}
 \label{fig:closure_condition}
\end{figure*}

\begin{remark}[Uniform death as an identifiability/closure condition (not a biological claim)]
The specialization $\delta(x)\equiv\delta$ should be read as a structural identifiability condition for exact closure, not as a literal assertion that TD mortality is damage-independent in vivo. The balance laws in Theorem~\ref{thm:balance-laws} show that the only obstruction to a closed two-dimensional dynamics on totals $(\bar P,\bar W)$ is the mortality flux $\int_0^\infty \delta(x)W(t,x)\,dx$. Exact closure on $(\bar P,\bar W)$ occurs precisely when this flux can be written as a function of $\bar W(t)$ alone, i.e. when $\int_0^\infty \delta(x)W(t,x)\,dx=\delta\,\bar W(t)$, yielding the planar ODE system in Corollary~\ref{cor:uniform-death-closure}. Conversely, when $\delta(x)$ varies with damage, the flux depends on the shape of $W(t,\cdot)$ rather than its total mass, breaking the reduction (Fig.\ref{fig:closure_condition}) and implying that any low-dimensional closure must include additional distributional information (e.g. weighted moments). Importantly, this condition is identifiable at the level of totals: if aggregate trajectories $(\bar P(t),\bar W(t))$ fail to satisfy the closed balance laws \eqref{eq:reduced_balance} beyond discretization error (Section~\ref{subsubsec:balance_residual}, Fig.~\ref{fig:balance_residual}), then the uniform-death closure is empirically falsified for that regime.
\end{remark}

This observation motivates the uniform-death assumption $\delta(x)\equiv\delta$ as a
structural hypothesis under which the PDE-to-ODE reduction is exact and analytically tractable.
Importantly, uniform death is also identifiable at the level of totals: it is precisely the
condition under which the mortality flux can be written as $\delta\,\bar W(t)$, yielding a closed
two-dimensional system whose equilibrium laws and global dynamics can be tested against
time-series measurements of aggregate stem/TD abundances.

\section{Replicator mapping and game-theoretic interpretation}
\label{sec:replicator_game}
In this section we show that the exact two-compartment closure obtained under uniform TD death
(Corollary~\ref{cor:uniform-death-closure}) induces a replicator-type frequency dynamics.
This provides an organizing interpretation in which strategies correspond to phenotypes
(stem vs.\ TD), and payoffs are the phenotypes' per-capita net growth rates.
Importantly, the game-theoretic structure here is not postulated: it is derived mechanistically from
the PDE bookkeeping via the exact reduction. 
The induced game is state dependent: the ``payoffs'' are not constants from a payoff matrix, but per-capita growth rates generated by the underlying feedback-regulated lineage dynamics.
Replicator dynamics are standard in evolutionary game
theory and population dynamics \citep{schuster_replicator_1983,hofbauer_evolutionary_1998,cressman_replicator_2014}.

\subsection{Frequency dynamics}
\label{subsec:frequency_dynamics}
Assume $\delta(x)\equiv\delta>0$ so that the totals $(\bar P,\bar W)$ satisfy the closed ODE system
\eqref{eq:reduced_balance} on the forward-invariant quadrant $\bar P,\bar W\ge 0$.
For interior states $\bar P,\bar W>0$, define the per-capita net growth rates
\begin{equation}\label{eq:percap_payoffs}
F_P(\bar P,\bar W)\coloneqq \frac{\dot{\bar P}}{\bar P},\qquad
F_W(\bar P,\bar W)\coloneqq \frac{\dot{\bar W}}{\bar W}.
\end{equation}
These quantities play the role of state-dependent payoffs in the induced two-strategy game.
  
Let $M(t):=\bar P(t)+\bar W(t)$ denote the total mass and define the stem frequency
\begin{equation}\label{eq:stem_frequency}
s(t):=\frac{\bar P(t)}{M(t)}\in(0,1),\qquad 1-s(t)=\frac{\bar W(t)}{M(t)}.
\end{equation}
A direct computation yields the replicator form
\begin{equation}\label{eq:replicator_form}
\dot s
=\frac{\dot{\bar P}}{M}-\frac{\bar P}{M^2}\dot M
=\frac{\bar P}{M}\Big(1-\frac{\bar P}{M}\Big)\Big(\frac{\dot{\bar P}}{\bar P}-\frac{\dot{\bar W}}{\bar W}\Big)
=s(1-s)\big(F_P-F_W\big),
\end{equation}
which is the familiar two-strategy replicator equation with payoffs given by per-capita growth
\citep{hofbauer_evolutionary_1998,cressman_replicator_2014}.
Equation \eqref{eq:replicator_form} shows that the sign of $F_P-F_W$ determines the direction of
selection on stemness: if stem cells have higher per-capita net growth ($F_P>F_W$), then $s$
increases; otherwise $s$ decreases.

\subsection{Nash/homeostasis equivalence}
\label{subsec:nash_homeostasis}
For two-strategy replicator dynamics, interior rest points correspond to payoff equalization across
strategies in support \citep{hofbauer_evolutionary_1998}. In our setting this equalization
has a concrete biological meaning: at homeostasis, stem and TD phenotypes have identical per-capita
net growth, so neither phenotype has a selective advantage in the aggregated (well-mixed) dynamics.
In \eqref{eq:percap_payoffs}, we interpret $F_P$, $F_W$ as state-dependent payoffs of an induced two-strategy game.

\begin{theorem}[Interior equilibrium $\Longleftrightarrow$ payoff equalization]
\label{thm:nash_equiv}

Assume $\delta(x)\equiv\delta>0$ and let $(\bar P,\bar W)\in(0,\infty)^2$. Then,
\begin{enumerate}[label=(\roman*)]
\item $(\bar P, \bar W)$ is an equilibrium if and only if the per-capita rates are equal and the common per-capita rate vanishes; equivalently,
\[
(\bar P, \bar W) \text{ is an equilibrium} \quad \Longleftrightarrow \quad F_P(\bar P,\bar W) = F_W(\bar P,\bar W)= 0.
\]
\item Separately, $F_P=F_W$ (on an interior state) is equivalent to the frequency $s = \bar P/(\bar P + \bar W)$ being a rest point of the replicator equation $\dot s = s(1-s)(F_P-F_W)$. 
\end{enumerate}
\end{theorem}

\begin{proof}
(i) If $(\bar P,\bar W)$ is an equilibrium then $\dot{\bar P}=\dot{\bar W}=0$. By the definitions
\[
F_P(\bar P,\bar W)=\frac{\dot{\bar P}}{\bar P},\qquad 
F_W(\bar P,\bar W)=\frac{\dot{\bar W}}{\bar W},
\]
we obtain $F_P(\bar P,\bar W)=F_W(\bar P,\bar W)=0$.

Conversely, if $F_P(\bar P,\bar W)=F_W(\bar P,\bar W)=0$ then
$\dot{\bar P}=F_P\bar P=0$ and $\dot{\bar W}=F_W\bar W=0$, so $(\bar P,\bar W)$ is an equilibrium.
(We remark that $F_P=F_W$ alone only implies the existence of a common per-capita rate $c$ with
$\dot{\bar P}=c\bar P,\ \dot{\bar W}=c\bar W$, and the additional requirement $c=0$ is necessary
for a fixed point.)

(ii) Let $M=\bar P+\bar W$ and $s=\bar P/M$. By \eqref{eq:replicator_form}, on the interior $s\in(0,1)$ we have $\dot s=0$ if and only if $F_P=F_W$, as claimed.
\end{proof}

In evolutionary game theory, a mixed Nash equilibrium equalizes the expected payoffs of strategies
in its support; otherwise a player could unilaterally switch to a higher-payoff strategy
\citep{hofbauer_evolutionary_1998}. Theorem~\ref{thm:nash_equiv} shows that in our model,
homeostasis is exactly such a payoff equalization condition, with payoffs derived from the
mechanistic division--differentiation--dedifferentiation--death bookkeeping.  

\subsection{Closed-form laws and experimental observables}
\label{subsec:closed_form_laws}
A key advantage of the present reduction is that payoff equalization can be rewritten as explicit
algebraic identities among measurable quantities. Under uniform death, the interior equilibrium
relations can be expressed in closed form.

\begin{theorem}[Homeostatic laws under uniform death]
\label{thm:equalization_ratio_laws}
Assume $\delta(x)\equiv\delta>0$. An interior equilibrium $(\bar P^*,\bar W^*)\in(0,\infty)^2$ of
\eqref{eq:reduced_balance} exists if and only if the following identities hold:
\begin{align}
\label{eq:equalization_law}
\big(p_2(\bar W^*)-p_1(\bar W^*)\big)\,\delta &= \lambda_R(\bar P^*),
\qquad\text{(\emph{Equalization law})}\\
\label{eq:ratio_law}
\frac{\bar P^*}{\bar W^*} &= \frac{\delta}{\lambda_P(\bar W^*)}.
\qquad\text{(\emph{Ratio law})}
\end{align}
At such an equilibrium, $F_P(\bar P^*,\bar W^*)=F_W(\bar P^*,\bar W^*)$ (payoff equalization),
and the induced frequency $s^*=\bar P^*/(\bar P^*+\bar W^*)$ is a rest point of
\eqref{eq:replicator_form}.
\end{theorem}

\begin{proof}
At equilibrium, $\dot{\bar P}=\dot{\bar W}=0$ in \eqref{eq:reduced_balance}. From $\dot{\bar P}=0$,
\[
\lambda_R(\bar P^*)\,\bar W^*=\big(p_2(\bar W^*)-p_1(\bar W^*)\big)\lambda_P(\bar W^*)\,\bar P^*.
\]
From $\dot{\bar W}=0$,
\[
\big(\delta+\lambda_R(\bar P^*)\big)\bar W^*
=\big(1-p_1(\bar W^*)+p_2(\bar W^*)\big)\lambda_P(\bar W^*)\,\bar P^*.
\]
Dividing the second identity by the first yields \eqref{eq:equalization_law}; substituting back
gives \eqref{eq:ratio_law}. The remaining statements follow from
Theorem~\ref{thm:nash_equiv}.
\end{proof}

Figure~\ref{fig:nash_geometry} illustrates the geometry of the homeostasis state, including the Nash equilibrium $(\bar P^*,\bar W^*)$ as the intersection of the Ratio law and the Equalization law (Theorem~\ref{thm:equalization_ratio_laws}), and the replicator dynamics with respect to the stem frequency $s$.

\begin{figure*}[htbp]
 \centering
 \includegraphics[width=\linewidth]{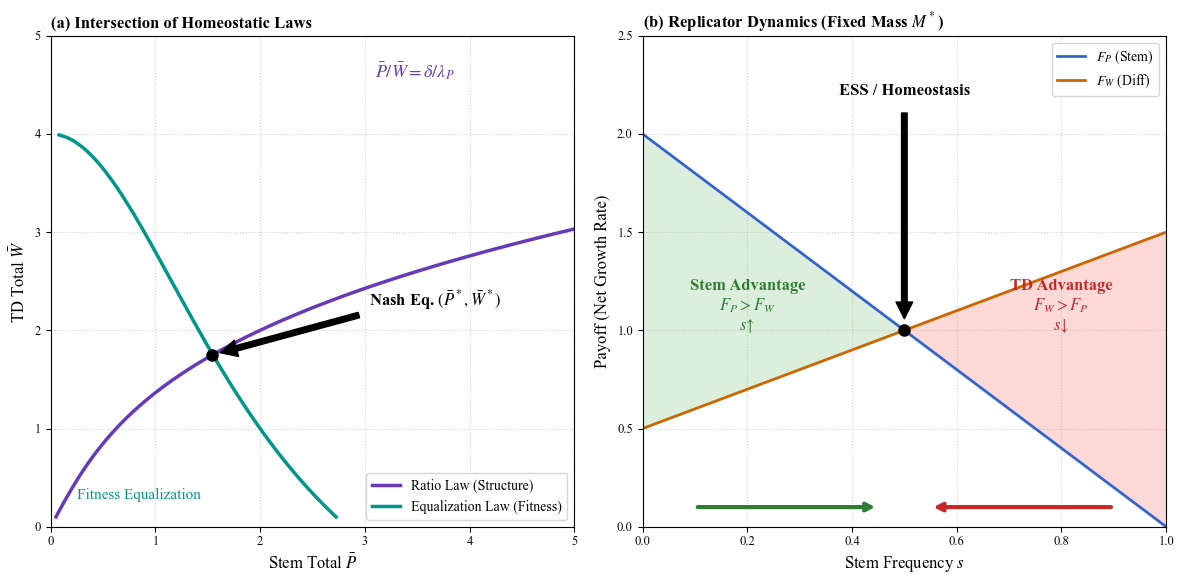} 
 \caption{\textbf{The geometry of homeostasis.}
 \textbf{(a) Intersection of laws:} The unique Nash equilibrium is determined by the intersection of the Ratio Law (purple) and Equalization Law (teal) curves.
 \textbf{(b) Replicator stability:} The Nash point is a stable interior rest point of the induced replicator dynamics in the parameter regime shown. Green/red zones indicate regions where stem/TD phenotypes have a fitness advantage ($F_P > F_W$ or $F_W > F_P$), driving the frequency $s$ toward equilibrium.}
 \label{fig:nash_geometry}
\end{figure*}

Equations \eqref{eq:equalization_law}--\eqref{eq:ratio_law} are algebraic relations among aggregate observables and effective rates at homeostasis:
\begin{enumerate}[label=(\roman*)]
\item $\bar P^*$ and $\bar W^*$ are steady-state compartment sizes, measurable by cell counting,
flow cytometry, or lineage-marker gating (stem vs.\ TD markers) \citep{vadakke-madathil_flow_2019,bose_pluripotent_2023}.
\item $\delta$ is the TD death rate; in steady conditions it can be estimated by apoptosis/necrosis
assays (e.g.\ Annexin V/PI) or survival analysis of labeled TD cohorts \citep{crowley_quantitation_2016,chen_estimating_2009}.
\item $\lambda_P(\bar W^*)$ is the effective stem division rate at TD load $\bar W^*$; it can be
estimated from proliferation readouts (EdU/BrdU incorporation, mitotic index, or time-lapse lineage
tracking) conditioned on the observed steady $\bar W^*$ \citep{huelsz-prince_mother_2022}.
\item $\lambda_R(\bar P^*)$ is the effective dedifferentiation rate at stem load $\bar P^*$; it can
be inferred from lineage tracing that quantify TD $\to$ stem return flux
under steady-state conditions \citep{weinberg_lineage_2007}.
\end{enumerate}
Thus, the equalization and Ratio laws provide directly testable predictions: once $\delta$ and
$\lambda_P(\bar W^*)$ are estimated, the steady-state ratio $\bar P^*/\bar W^*$ is predicted by
\eqref{eq:ratio_law}, and the dedifferentiation rate required to offset net differentiation losses
is predicted by \eqref{eq:equalization_law}.

\section{Dynamics: thresholds, uniqueness, and global stability}
\label{sec:dynamics}
Throughout this section we assume the uniform TD death regime $\delta(x)\equiv\delta>0$,
so that the PDE balance laws close exactly to the planar ODE system on totals
$(\bar P,\bar W)$ (Corollary~\ref{cor:uniform-death-closure}):
\begin{equation}\label{eq:reduced_balance_recall}
\begin{cases}
\dot{\bar P}
=\Big(p_1(\bar W)-p_2(\bar W)\Big)\lambda_P(\bar W)\,\bar P
+\lambda_R(\bar P)\,\bar W,\\[6pt]
\dot{\bar W}
=\Big(1-p_1(\bar W)+p_2(\bar W)\Big)\lambda_P(\bar W)\,\bar P
-\big(\delta+\lambda_R(\bar P)\big)\,\bar W.
\end{cases}
\end{equation}
The closed quadrant $\{\bar P\ge0,\ \bar W\ge0\}$ is forward invariant and solutions are global
(Corollary~\ref{cor:uniform-death-closure}). Our aim is to expose the global qualitative structure:
a structural threshold at extinction, uniqueness of the interior equilibrium via a one-dimensional
reduction, and global non-oscillatory convergence under compensatory feedback.

\subsection{Extinction--growth threshold at the origin}
\label{subsec:threshold_origin}
We analyze the extinction state $(\bar P,\bar W)=(0,0)$ by open-loop linearization.
Let
\[
\hat p_i:=p_i(0),\qquad
\hat\lambda_P:=\lambda_P(0),\qquad
\hat\lambda_R:=\lambda_R(0),\qquad
\Delta p:=\hat p_1-\hat p_2 .
\]
Then the Jacobian of \eqref{eq:reduced_balance_recall} at the origin is
\begin{equation}\label{eq:Jac_origin}
A(\Delta p)=
\begin{pmatrix}
\Delta p\,\hat\lambda_P & \hat\lambda_R\\
(1-\Delta p)\,\hat\lambda_P & -(\delta+\hat\lambda_R)
\end{pmatrix},
\end{equation}
with trace and determinant
\begin{equation}\label{eq:tr_det_origin}
\mathrm{tr}\,A(\Delta p)=\Delta p\,\hat\lambda_P-(\delta+\hat\lambda_R),\qquad
\det A(\Delta p)=-\delta\,\hat\lambda_P\Big(\Delta p+\frac{\hat\lambda_R}{\delta}\Big).
\end{equation}

\begin{theorem}[Structural extinction--growth threshold]
\label{thm:extinction_growth}
Define
\begin{equation}\label{eq:dpcrit}
\Delta p_{\mathrm{crit}}=-\frac{\hat\lambda_R}{\delta}.
\end{equation}
Then:
\begin{enumerate}[label=(\roman*), leftmargin=2.2em]
\item If $\Delta p<\Delta p_{\mathrm{crit}}$, then $\det A>0$ and $\mathrm{tr}\,A<0$, hence the origin is locally asymptotically stable (extinction regime).
\item If $\Delta p=\Delta p_{\mathrm{crit}}$, then $\det A=0$ and one eigenvalue is $0$ (threshold case).
\item If $\Delta p>\Delta p_{\mathrm{crit}}$, then $\det A<0$, hence the origin is a saddle with a one-dimensional unstable manifold (growth regime).
\end{enumerate}
\end{theorem}

\begin{proof}
The transition occurs when $\det A(\Delta p)=0$, i.e.\ $\Delta p=-\hat\lambda_R/\delta$ by
\eqref{eq:tr_det_origin}. The sign patterns in \eqref{eq:tr_det_origin} give the classification.
\end{proof}

Equation \eqref{eq:dpcrit} is a structural threshold: it depends only on (i) uniform death
$\delta$, and (ii) open-loop renewal--differentiation mismatch $\Delta p$, and it follows from the
sign change of $\det A$. This makes \eqref{eq:dpcrit} robust to the detailed shape of feedback maps,
and directly interpretable as the dividing line between washout/extinction and departure from
extinction.
Explicit feedback-controlled thresholds are also a recurring outcome in other coupled continuum/hybrid systems, e.g., finite-domain instability thresholds in bidirectionally coupled endothelial–factor dynamics \citep{liu_bidirectional_2025}.

\subsection{Uniqueness of the interior equilibrium via a one-dimensional reduction}
\label{subsec:uniqueness_1d}
At any interior equilibrium $(\bar P^*,\bar W^*)\in(0,\infty)^2$, the closed-form laws from
Section~\ref{subsec:closed_form_laws} hold:
\begin{equation}\label{eq:equil_laws_recall}
\lambda_R(\bar P^*)=\big(p_2(\bar W^*)-p_1(\bar W^*)\big)\delta,\qquad
\frac{\bar P^*}{\bar W^*}=\frac{\delta}{\lambda_P(\bar W^*)}.
\end{equation}
The Ratio law defines the \emph{ratio curve}
\begin{equation}\label{eq:P_of_W}
\bar P=\bar P(\bar W)\coloneqq \frac{\delta}{\lambda_P(\bar W)}\,\bar W,\qquad \bar W>0.
\end{equation}
Substituting \eqref{eq:P_of_W} into the Equalization law yields a scalar root condition.

\medskip\noindent\textbf{One-dimensional residual.}
Assume $\lambda_R$ is strictly decreasing (hence invertible on its range) and define
\begin{equation}\label{eq:G_def}
G(\bar W) \coloneqq \lambda_R^{-1}\!\Big(\delta\big(p_2(\bar W)-p_1(\bar W)\big)\Big)
-\frac{\delta}{\lambda_P(\bar W)}\,\bar W,\qquad \bar W>0.
\end{equation}
Then $\bar W^*>0$ solves $G(\bar W^*)=0$ if and only if
$\bar P^*=\bar P(\bar W^*)$ and $(\bar P^*,\bar W^*)$ is an interior equilibrium of
\eqref{eq:reduced_balance_recall}.

\begin{theorem}[Uniqueness under a verifiable slope condition]
\label{thm:unique_equilibrium_strong}
Assume $p_1,p_2,\lambda_P\in C^1([0,\infty))$ with $\lambda_P(\bar W)>0$ and $\lambda_P'(\bar W)<0$,
and $\lambda_R\in C^1([0,\infty))$ with $\lambda_R'(\bar P)<0$.
Define $\Delta(\bar W) \coloneqq p_2(\bar W)-p_1(\bar W)$ and $\bar P(\bar W)$ by \eqref{eq:P_of_W}.
Suppose $\Delta(\bar W)\ge 0$ for $\bar W\ge 0$ and $\Delta'(\bar W)<0$ on the relevant range.
If the slope dominance condition
\begin{equation}\label{eq:slope_dominance}
-\lambda_R'\big(\bar P(\bar W)\big)\,\bar P'(\bar W)\;<\;-\delta\,\Delta'(\bar W)
\qquad\text{for all }\bar W>0
\end{equation}
holds, then the function
\begin{equation}\label{eq:r_def}
r(\bar W) \coloneqq \delta\,\Delta(\bar W)-\lambda_R\big(\bar P(\bar W)\big)
\end{equation}
is strictly decreasing on $(0,\infty)$ and has at most one root. Consequently,
\eqref{eq:reduced_balance_recall} admits at most one interior equilibrium.
\end{theorem}

\begin{proof}
Differentiate \eqref{eq:r_def} using the chain rule:
\[
r'(\bar W)=\delta\,\Delta'(\bar W)-\lambda_R'\big(\bar P(\bar W)\big)\,\bar P'(\bar W).
\]
Under \eqref{eq:slope_dominance}, $r'(\bar W)<0$ for all $\bar W>0$, hence $r$ is strictly
decreasing and can cross $0$ at most once. Finally, $r(\bar W)=0$ is equivalent to the equalization
law evaluated on the ratio curve \eqref{eq:P_of_W}, hence to interior equilibria.
\end{proof}

\paragraph{Computable form of \eqref{eq:slope_dominance}}
From \eqref{eq:P_of_W},
\begin{equation}\label{eq:Pprime}
\bar P'(\bar W)=\delta\frac{\lambda_P(\bar W)-\bar W\,\lambda_P'(\bar W)}{\lambda_P(\bar W)^2}>0
\quad(\text{since }\lambda_P>0,\ \lambda_P'<0).
\end{equation}
Thus \eqref{eq:slope_dominance} compares a dedifferentiation sensitivity term
$-\lambda_R'(\bar P)\bar P'(\bar W)$ against a differentiation-gap sensitivity term
$-\delta\Delta'(\bar W)$, and can be checked numerically for any given feedback parameterization
(e.g.\ Hill families).

\paragraph{A weaker (nontriviality) uniqueness statement}
Even without \eqref{eq:slope_dominance}, if at least one of $\Delta(\cdot)$ or $\lambda_R(\cdot)$ is
non-constant and strictly monotone, then the pair of algebraic laws \eqref{eq:equil_laws_recall}
generically determines at most one interior equilibrium (cf.\ the one-dimensional solver in
Section~\ref{sec:numerics}), and Theorem~\ref{thm:unique_equilibrium_strong} provides a clean,
verifiable sufficient condition.

\subsection{Compensatory feedback and non-oscillation (global structure)}
\label{subsec:comp_feedback_global}
We now identify a sign/monotonicity regime enforcing global non-oscillatory dynamics in the
interior of the quadrant.

\begin{definition}[Compensatory feedback]\label{def:comp_feedback}
We say the feedback is \emph{compensatory} if
\begin{enumerate}[label=(\roman*), leftmargin=2.2em]
\item $p_1'(\bar W)>p_2'(\bar W)$ for all $\bar W\ge 0$ (in particular $p_1',p_2'<0$),
equivalently $\Delta'(\bar W)=p_2'(\bar W)-p_1'(\bar W)<0$;
\item $p_1(\bar W)<p_2(\bar W)$ on the relevant range, equivalently $\Delta(\bar W)>0$.
\end{enumerate}
\end{definition}

\begin{theorem}[Negative divergence and exclusion of periodic orbits]
\label{thm:no_cycles}
Assume $\delta>0$ is constant and that $p_1,p_2,\lambda_P$ are $C^1$, bounded, strictly decreasing
in $\bar W$, while $\lambda_R$ is $C^1$, bounded, strictly decreasing in $\bar P$.
If Definition~\ref{def:comp_feedback} holds, then the divergence of the vector field
$\mathbf G=(G_1,G_2)$ in \eqref{eq:reduced_balance_recall} satisfies
\begin{equation}\label{eq:neg_div}
\nabla\!\cdot\!\mathbf G(\bar P,\bar W)
=\partial_{\bar P}G_1(\bar P,\bar W)+\partial_{\bar W}G_2(\bar P,\bar W)<0
\qquad\forall\,(\bar P,\bar W)\in(0,\infty)^2.
\end{equation}
Consequently, by the Bendixson--Dulac criterion (take Dulac function $\varphi\equiv 1$),
\eqref{eq:reduced_balance_recall} admits no nontrivial periodic orbit contained in $(0,\infty)^2$.
\end{theorem}

\begin{proof}
Compute
\[
\partial_{\bar P}G_1(\bar P,\bar W)
=\big(p_1(\bar W)-p_2(\bar W)\big)\lambda_P(\bar W)+\lambda_R'(\bar P)\bar W<0
\]
since $p_1-p_2<0$ and $\lambda_R'<0$.
Next,
\[
\partial_{\bar W}G_2(\bar P,\bar W)
=\Big[(-p_1'(\bar W)+p_2'(\bar W))\lambda_P(\bar W)
+(1-p_1(\bar W)+p_2(\bar W))\lambda_P'(\bar W)\Big]\bar P-(\delta+\lambda_R(\bar P)).
\]
Under compensatory feedback, $-p_1'+p_2'=\Delta'(\bar W)<0$; moreover $\lambda_P'(\bar W)<0$ and
$1-p_1+p_2>0$, hence the bracketed term is negative, so $\partial_{\bar W}G_2<0$.
Summing yields \eqref{eq:neg_div}. The Bendixson--Dulac criterion \citep{perko_differential_2001} applies on any simply connected
subset of $(0,\infty)^2$ to exclude periodic orbits.
\end{proof}

\begin{theorem}[Global convergence among bounded trajectories]
\label{thm:global_convergence}
Assume the hypotheses of Theorem~\ref{thm:no_cycles}. If an interior equilibrium exists and is
unique (e.g.\ by Theorem~\ref{thm:unique_equilibrium_strong}), then every bounded trajectory of \eqref{eq:reduced_balance_recall}
with initial data in $(0,\infty)^2$ converges to that equilibrium as $t\to\infty$.
\end{theorem}

\begin{proof}
Let $(\bar P(t),\bar W(t))$ be a bounded trajectory in $(0,\infty)^2$.
Its $\omega$-limit set is nonempty, compact, and invariant. By the Poincar\'e--Bendixson theorem,
the $\omega$-limit set is either an equilibrium, a periodic orbit, or a union of equilibria with
connecting orbits. Theorem~\ref{thm:no_cycles} excludes periodic orbits, hence the $\omega$-limit set
consists of equilibria (possibly with heteroclinic connections). Uniqueness of the interior equilibrium
forces the $\omega$-limit set to be that single point, i.e.\ convergence.
\end{proof}

\paragraph{Global structure (not merely local stability)}
Theorems~\ref{thm:no_cycles}--\ref{thm:global_convergence} constrain the entire phase portrait
in the interior: oscillations are ruled out and bounded dynamics must settle to equilibrium. This is
strictly stronger than a local eigenvalue computation at $(\bar P^*,\bar W^*)$.

\subsection{Scaling invariance under uniform gain amplification}
\label{subsec:scaling_invariance}
Finally, the Hill feedback family \eqref{eq:Hill_maps} admits a scaling symmetry that separates composition from scale. For $A>0$, rescale all Hill gains by $k_i\mapsto A k_i$ and
denote the corresponding maps by $p_{j,A},\lambda_{P,A},\lambda_{R,A}$. Then
$p_{j,A}(w/A)=p_j(w)$ and $\lambda_{P,A}(w/A)=\lambda_P(w)$, while
$\lambda_{R,A}(p/A)=\lambda_R(p)$.

\begin{proposition}[Gain-scaling symmetry]\label{prop:gain_scaling}
Assume uniform death $\delta>0$ and Hill-type feedback \eqref{eq:Hill_maps}.
If $(\bar P^*,\bar W^*)$ is an interior equilibrium of \eqref{eq:reduced_balance_recall}, then
\begin{equation}\label{eq:scaling_map}
(\bar P^*,\bar W^*)\longmapsto \Big(\frac{\bar P^*}{A},\frac{\bar W^*}{A}\Big)
\end{equation}
is an interior equilibrium of the gain-scaled system. Moreover, the equilibrium ratio
$\bar P^*/\bar W^*$ is invariant under this transformation.
\end{proposition}

\begin{proof}[Proof sketch]
The equilibrium is characterized by the two algebraic laws
\eqref{eq:equalization_law}--\eqref{eq:ratio_law}. Under the gain scaling,
evaluating the scaled maps at $(\bar P^*/A,\bar W^*/A)$ reproduces the original values at
$(\bar P^*,\bar W^*)$, hence both laws remain satisfied. Ratio invariance is immediate.
\end{proof}

Uniform gain amplification changes the level of homeostasis but not the mix:
the stem-to-TD ratio remains fixed while totals rescale by $1/A$. In the game-theoretic view
(Section~\ref{sec:replicator_game}), this corresponds to a uniform amplification of feedback
sensitivity that rescales absolute abundances while preserving the Nash/ESS composition. Figure~\ref{fig:dynamics_structure} visualizes the extinction-growth threshold (Theorem~\ref{thm:extinction_growth}) and the gain-scaling symmetry (Proposition~\ref{prop:gain_scaling}).

\begin{figure*}[htbp]
 \centering
 \includegraphics[width=\textwidth]{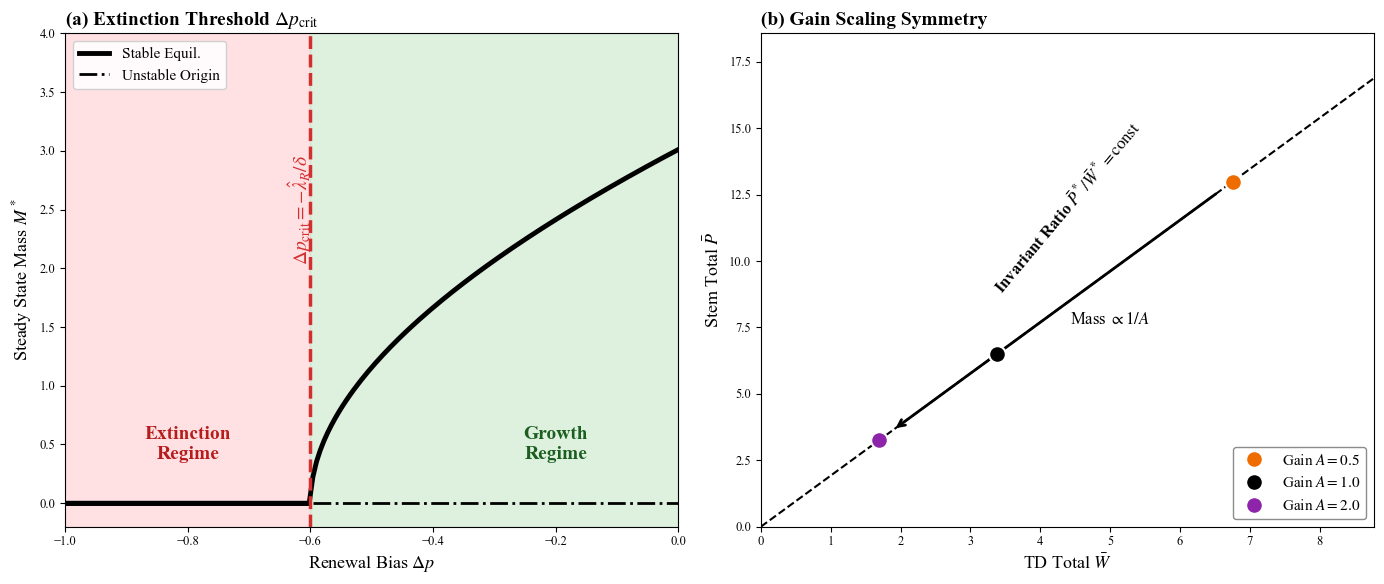} 
 \caption{\textbf{Global qualitative structure and scaling symmetry.}
 \textbf{(a) Extinction-growth threshold (Theorem 5.1):} A stability threshold at the origin separates the extinction regime (red zone, stable origin) from the growth regime (green zone, stable interior equilibrium). The critical threshold occurs exactly at $\Delta p_{\mathrm{crit}}$.
 \textbf{(b) Gain-scaling symmetry (Proposition 5.6):} Under uniform amplification of feedback gains ($A \to 2A$), the equilibrium point shifts along a ray of constant stem-to-TD ratio (dashed line), demonstrating that tissue composition is invariant to global sensitivity scaling while total mass scales as $1/A$.}
 \label{fig:dynamics_structure}
\end{figure*}

\section{Numerical methods and reproducible verification}
\label{sec:numerics}

\begin{table}[h]
\centering
\caption{Correspondence between analytical results and certified numerical checks.}
\label{tab:evidence_chain}
\begin{tabular}{lll}
\hline
\textbf{Theorem / Result} & \textbf{Numerical check} & \textbf{Section / Figure} \\
\hline
Theorem~\ref{thm:balance-laws} (balance laws) &
  Certified residual bounds $R_P,R_W,R_M$ &
  \ref{subsubsec:balance_residual}, Fig.~\ref{fig:balance_residual} \\
Corollary~\ref{cor:uniform-death-closure} (ODE closure) &
  PDE totals vs.\ ODE &
  \ref{subsubsec:pde_vs_ode}, Fig.~\ref{fig:pde_vs_ode} \\
Theorem~\ref{thm:nash_equiv} (Nash = homeostasis) &
  Replicator field zero-crossing &
  \ref{subsubsec:replicator_evidence}, Fig.~\ref{fig:replicator} \\
Theorem~\ref{thm:equalization_ratio_laws} (equalization/Ratio laws) &
  Certified equilibrium residuals &
  \ref{subsubsec:nash_cert} \\
Theorem~\ref{thm:extinction_growth} (threshold) &
  Bifurcation scan &
  \ref{subsubsec:bifurcation}, Fig.~\ref{fig:bifurcation} \\
Theorem~\ref{thm:no_cycles} (Dulac / no cycles) &
  Divergence field sign check &
  \ref{subsubsec:divergence}, Fig.~\ref{fig:divergence} \\
Theorem~\ref{thm:global_convergence} (global convergence) &
  Phase portrait convergence &
  \ref{subsubsec:phase_portrait}, Fig.~\ref{fig:phase_portrait} \\
Proposition~\ref{prop:gain_scaling} (gain scaling) &
  Scaling invariance to $\varepsilon_{\mathrm{mach}}$ &
  \ref{subsubsec:scaling}, Table~\ref{tab:scaling} \\
\hline
\end{tabular}
\end{table}

This section provides a self-contained computational counterpart to the analytical results
of Sections~\ref{sec:balance_reduction}--\ref{sec:dynamics}.
Following best practice for certified computational verification in mathematical biology
\citep{oberkampf_verification_2002,roache_verification_1998},
we organize the material into two layers.
The first is the verification layer. 
Do the discrete schemes solve the stated equations accurately?
The second is the certified computational check layer. 
Do the numerical solutions exhibit the qualitative
behavior predicted by the theorems, with all reported quantities programmatically emitted
by deterministic, parameter-locked scripts?
Table~\ref{tab:evidence_chain} maps each main theorem to the specific verification test or certified computational check that supports it. 
Related recent work also stresses algebraic–spectral stability tests for explicit time-stepping schemes (e.g., forward Euler) alongside closed-form thresholds \citep{wang_algebraicspectral_2026}, and it provides a complementary verification paradigm for mechanistic dynamical models.

\subsection{Numerical schemes}
\label{subsec:schemes}

Our PDE solver is the first-order upwind finite-volume scheme with operator splitting. We discretize the damage-structured system \eqref{eq:PDE_model} on a uniform grid
$\{x_j=\bigl(j+\tfrac12\bigr)\Delta x\}_{j=0}^{N_x-1}$ over $[0,x_{\max}]$
with $\Delta x=x_{\max}/N_x$ and 
$x_{\max}=10$ used for baseline runs; for closure-verification and long-horizon comparisons we increase $x_{\max}$ (Section~\ref{subsubsec:pde_vs_ode}).

Since $v_P,v_W>0$, information propagates in the positive-$x$ direction and the upwind
numerical flux at the interface $x_{j+1/2}$ is
\begin{equation}\label{eq:num_upwind_flux}
\mathcal F_{j+1/2}^P = v_P\, P_j^n, \qquad
\mathcal F_{j+1/2}^W = v_W\, W_j^n,
\end{equation}
with $\mathcal F_{1/2}^P=\mathcal F_{1/2}^W=0$ enforcing the homogeneous inflow boundary
condition $P(t,0)=W(t,0)=0$.
The conservative finite-volume update for the transport step is
\begin{equation}\label{eq:num_fv_transport}
P_j^{n+1,\mathrm{tr}} = P_j^n - \frac{\Delta t}{\Delta x}
\bigl(\mathcal F_{j+1/2}^P - \mathcal F_{j-1/2}^P\bigr),
\end{equation}
and analogously for $W_j^{n+1,\mathrm{tr}}$.

The birth operators $\mathcal B_P,\mathcal B_W$ (cf.~\eqref{eq:birth_ops})
involve the dilation (push-forward) map
$[\mathcal T_\alpha f](x) = (1/\alpha)\,f(x/\alpha)$
for each inheritance fraction
$\alpha\in\{\alpha_1,\alpha_2,\gamma_1\}$ (stem branch) or
$\alpha\in\{\beta_1,\beta_2,\gamma_2\}$ (TD branch).
The fundamental property driving the balance-law reduction is the mass identity
$\int_0^\infty [\mathcal T_\alpha f]\,dx = \int_0^\infty f\,dx$
(Lemma~\ref{lem:cov_identity}). A discretization that does not preserve this identity
exactly would introduce a spurious mass defect that contaminates the balance-law residual
(Test~3) and the PDE-vs-ODE closure check.

We therefore use a conservative finite-volume projection (flux-form remapping)
rather than pointwise interpolation.
Treating $P_k$ as the cell average on $C_k=[k\Delta x,(k+1)\Delta x]$, the dilation maps
$C_k$ to the image interval $\alpha C_k=[\alpha k\Delta x,\alpha(k+1)\Delta x]$ of width
$\alpha\Delta x$.
Since $\alpha\in(0,1)$, this image is narrower than $\Delta x$ and overlaps at most two
target cells.
We redistribute the source mass $P_k\Delta x$ onto the target grid by exact overlap
fractions:
\begin{equation}\label{eq:num_conservative_remap}
[\mathcal T_{\alpha,h} P]_j
= \frac{1}{\alpha\,\Delta x}\sum_{k=0}^{N_x-1} P_k\;\omega_{jk}(\alpha),
\end{equation}
where $\omega_{jk}(\alpha):=|\alpha C_k\cap C_j|$ is the length of the overlap between
the image interval and target cell $C_j$.
Because $\sum_j \omega_{jk}(\alpha)=|\alpha C_k|=\alpha\,\Delta x$ for every source cell $k$
(provided $\alpha C_k\subset[0,x_{\max}]$), we obtain the exact discrete conservation
identity
\begin{equation}\label{eq:num_exact_conservation}
\sum_{j=0}^{N_x-1} [\mathcal T_{\alpha,h} P]_j\,\Delta x
= \sum_{k=0}^{N_x-1} P_k\,\Delta x
\end{equation}
to machine precision ($\varepsilon_{\mathrm{mach}}\approx 2.2\times10^{-16}$).
The overlap weights $\omega_{jk}$ are nonzero only for $j\in\{j_L(k),j_L(k)+1\}$ with $j_L(k) = \lfloor \alpha k\rfloor$, so the
operation is $\mathcal O(N_x)$ and implemented via vectorized scatter-add.

The full discrete birth operators are then assembled as in the continuous case:
\[
[\mathcal B_{P,h}]_j
= p_1(\bar W_h)\,\lambda_P(\bar W_h)
  \bigl([\mathcal T_{\alpha_1,h}P]_j+[\mathcal T_{\alpha_2,h}P]_j\bigr)
+ p_3(\bar W_h)\,\lambda_P(\bar W_h)\,[\mathcal T_{\gamma_1,h}P]_j,
\]
and analogously for $\mathcal B_{W,h}$.
By linearity of \eqref{eq:num_exact_conservation} and $p_1+p_2+p_3=1$, this yields the
exact discrete analogue of Lemma~\ref{lem:birth_mass}:
\[
\sum_j [\mathcal B_{P,h}+\mathcal B_{W,h}]_j\,\Delta x
= 2\,\lambda_P(\bar W_h)\,\bar P_h.
\]

We use first-order operator splitting: given $(P^n,W^n)$ at $t^n$,
we first advance the transport step by \eqref{eq:num_fv_transport} (and its analogue for $W$),
obtaining $(P^{n+1,\mathrm{tr}},W^{n+1,\mathrm{tr}})$.
We then update the reaction step by forward Euler,
\begin{align}\label{eq:num_fv_reaction}
P_j^{n+1} &= P_j^{n+1,\mathrm{tr}}
+ \Delta t\Big( \bigl[\mathcal B_{P,h}[P^{n+1,\mathrm{tr}}]\bigr]_j
- \lambda_P(\bar W_h^{n+1,\mathrm{tr}})\,P_j^{n+1,\mathrm{tr}}
+ \lambda_R(\bar P_h^{n+1,\mathrm{tr}})\,W_j^{n+1,\mathrm{tr}}\Big),\\
\nonumber
W_j^{n+1} &= W_j^{n+1,\mathrm{tr}}
+ \Delta t\Big( \bigl[\mathcal B_{W,h}[P^{n+1,\mathrm{tr}}]\bigr]_j
- \big(\delta+\lambda_R(\bar P_h^{n+1,\mathrm{tr}})\big)\,W_j^{n+1,\mathrm{tr}}\Big),
\end{align}
where $\bar P_h^{n+1,\mathrm{tr}}=\sum_j P_j^{n+1,\mathrm{tr}}\Delta x$ and
$\bar W_h^{n+1,\mathrm{tr}}=\sum_j W_j^{n+1,\mathrm{tr}}\Delta x$.
Time integration is performed under the CFL restriction
\begin{equation}\label{eq:num_CFL}
\Delta t \le \mathrm{CFL}\cdot\frac{\Delta x}{\max(v_P,v_W)},\qquad
\mathrm{CFL}=0.8.
\end{equation}

To suppress rare floating-point undershoots, we apply a positivity fix only when
$\min_j P_j^{n+1}< -10^{-14}$ or $\min_j W_j^{n+1}< -10^{-14}$, by setting negative values
to zero. We record the induced mass defect at each step and verify it remains below $10^{-12}$
relative to total mass in all runs reported here.

Because the scheme is first-order in both the transport upwinding and the operator splitting,
we expect $\mathcal{O}(\Delta x)$ convergence of the integrated quantities
$\bar P(t),\bar W(t)$; this is confirmed in
Section~\ref{subsubsec:grid_convergence}.

We now turn to the ODE solver. The exact two-compartment closure \eqref{eq:reduced_balance} is integrated using an adaptive Runge–Kutta method with error control.
The exact two-compartment closure \eqref{eq:reduced_balance} is integrated using the
Dormand--Prince RK4(5) pair (\textsc{RK45}), with relative tolerance $\varepsilon_{\mathrm{rel}}=10^{-10}$
and absolute tolerance $\varepsilon_{\mathrm{abs}}=10^{-12}$.
These stringent tolerances ensure that the ODE solution serves as a reference truth against
which PDE totals can be compared. For verification, we also cross-check against the
eighth-order Dormand--Prince method (\textsc{DOP853}) and the implicit Radau~IIA(5) solver;
agreement to $\mathcal O(10^{-10})$ confirms that temporal truncation error is negligible
(Section~\ref{subsubsec:timestep}).

Interior equilibria of \eqref{eq:reduced_balance} are characterised by the one-dimensional
root condition $r(\bar W)=0$ (cf.\ \eqref{eq:r_def}), where
\begin{equation}\label{eq:num_r_algo}
r(\bar W) \coloneqq \delta\,\bigl(p_2(\bar W)-p_1(\bar W)\bigr)
- \lambda_R\!\left(\frac{\delta\,\bar W}{\lambda_P(\bar W)}\right).
\end{equation}
The companion coordinate is recovered from the Ratio law:
$\bar P^* = \delta\,\bar W^*/\lambda_P(\bar W^*)$.

For the reference parameter set (Table~\ref{tab:params}), Algorithm~\ref{alg:nash}
detects exactly one sign change of $r(\bar W)$ and returns a unique interior equilibrium
$(\bar P^*,\bar W^*)$.
The Equalization law, Ratio law, and ODE steady-state residuals are all satisfied to
double-precision accuracy
(see Section~\ref{subsubsec:nash_cert}).

\begin{algorithm}[H]
\caption{Bracket-and-refine Nash solver}\label{alg:nash}
\begin{algorithmic}[1]
\Require Feedback maps $p_1,p_2,\lambda_P,\lambda_R$; death rate $\delta>0$;
         search interval $[\bar W_{\min},\bar W_{\max}]$; tolerance $\varepsilon$.
\Ensure  List of certified interior equilibria $\{(\bar P^*_k,\bar W^*_k)\}$.
\State Evaluate $r(\bar W)$ on a log-spaced grid of $N_{\mathrm{grid}}=2000$ points in
       $[\bar W_{\min},\bar W_{\max}]$.
\ForAll{consecutive pairs $(r_j,r_{j+1})$ with $r_j\,r_{j+1}<0$}
  \State Apply Brent's method on $[\bar W_j,\bar W_{j+1}]$ with tolerance $\varepsilon$ to
         obtain $\bar W^*$.
  \State Set $\bar P^*=\delta\,\bar W^*/\lambda_P(\bar W^*)$.
  \State Certify: evaluate both equilibrium laws \eqref{eq:equalization_law}--\eqref{eq:ratio_law}
         and the ODE right-hand side at $(\bar P^*,\bar W^*)$; accept if all residuals are
         $<10\varepsilon$.
\EndFor
\end{algorithmic}
\end{algorithm}

\begin{table}[t]
\centering
\caption{Reference parameter set for all numerical experiments.
The compensatory feedback conditions of Definition~\ref{def:comp_feedback}
are satisfied: $\Delta(\bar W):=p_2(\bar W)-p_1(\bar W)>0$ and
$\Delta'(\bar W)<0$ for all $\bar W\ge 0$
(verified numerically for the Hill families with these parameters).}
\label{tab:params}
\begin{tabular}{llc}
\hline
\textbf{Symbol} & \textbf{Description} & \textbf{Value} \\
\hline
$\hat p_1$        & Baseline self-renewal probability     & 0.20 \\
$\hat p_2$        & Baseline symmetric differentiation     & 0.40 \\
$\hat\lambda_P$   & Baseline stem division rate            & 1.0  \\
$\hat\lambda_R$   & Baseline dedifferentiation rate        & 0.3  \\
$k_1,k_2,k_3,k_4$ & Hill gains                            & 0.5  \\
$m_1,m_2,m_3,m_4$ & Hill exponents                        & 2    \\
$\delta$           & Uniform TD death rate                 & 0.5  \\
$v_P,v_W$          & Transport (damage accumulation) speeds & 1.0  \\
$\alpha_i,\beta_i,\gamma_i$ & Symmetric damage inheritance & 0.5  \\
\hline
\end{tabular}
\end{table}

All verification tests use the reference parameter set recorded in Table~\ref{tab:params}.
Initial data for the PDE are Gaussian bumps (rapidly decaying to working precision):
$P_0(x)=\exp\!\bigl(-(x-2)^2/0.5\bigr)$ and $W_0(x)=0.5\,\exp\!\bigl(-(x-1.5)^2/0.5\bigr)$.

\subsection{Test~1: PDE grid convergence}
\label{subsubsec:grid_convergence}
We run the PDE solver on four successively refined grids
$N_x\in\{100,200,400,800\}$ with $T=5$ and compare the computed total masses
$\bar P(t),\bar W(t)$ against the finest grid. Table~\ref{tab:grid_conv} reports the
$L^\infty$-norm of the inter-grid differences and the estimated convergence rates. Figure~\ref{fig:grid_convergence} shows the consistency of the stem and TD mass across different spatial resolutions.

\begin{table}[h]
\centering
\caption{Grid convergence of PDE totals ($T=5$, CFL${}=0.8$).
Differences are measured in $L^\infty([0,T])$ between consecutive refinements.}
\label{tab:grid_conv}
\begin{tabular}{cccccc}
\hline
$N_x^{\mathrm{coarse}}$ & $N_x^{\mathrm{fine}}$ &
$\|\Delta\bar P\|_\infty$ & $\|\Delta\bar W\|_\infty$ &
rate ($\bar P$) & rate ($\bar W$) \\
\hline
100 & 200 & $1.80\times 10^{-3}$ & $6.19\times 10^{-3}$ & --- & --- \\
200 & 400 & $8.69\times 10^{-4}$ & $2.97\times 10^{-3}$ & 1.05 & 1.06 \\
400 & 800 & $4.27\times 10^{-4}$ & $1.46\times 10^{-3}$ & 1.02 & 1.03 \\
\hline
\end{tabular}
\end{table}

\begin{figure}[htbp]
    \centering
    \includegraphics[width=\linewidth]{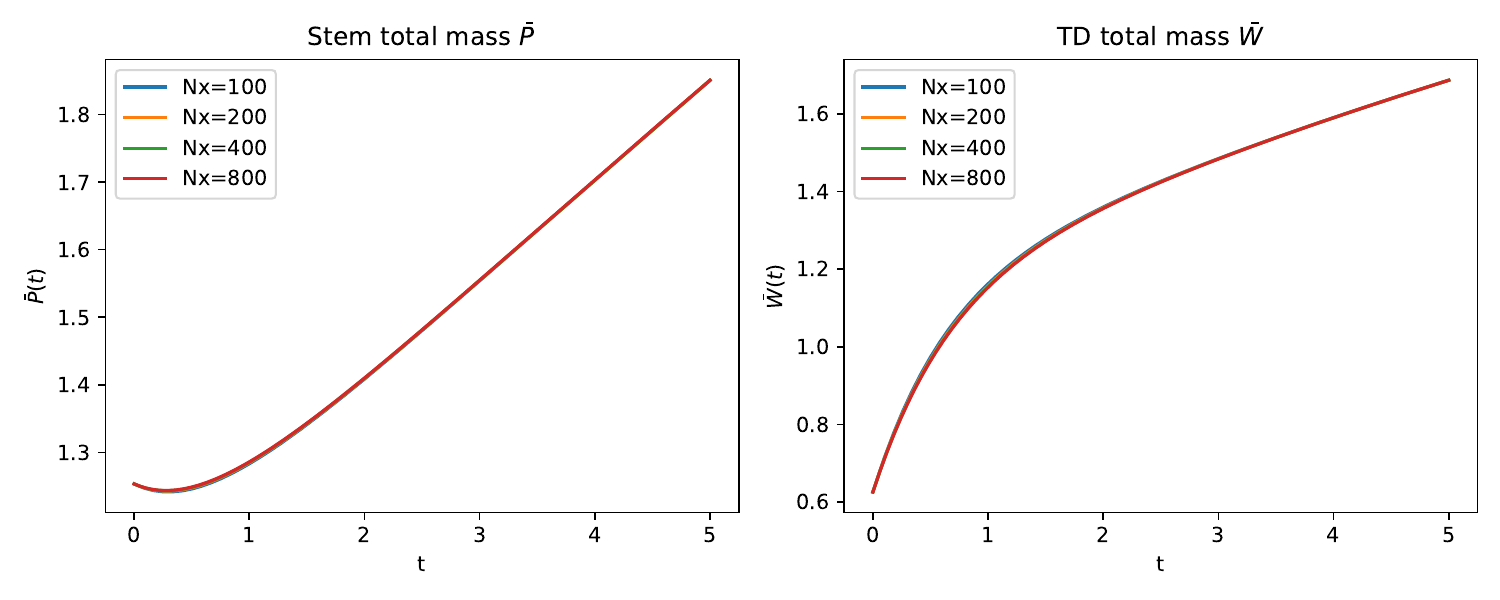}
    \caption{\textbf{Stem and TD total mass consistency under grid refinement tests}. \textbf{(a) Stem mass consistency:} Time evolution of the total stem population $\bar P(t)$ simulated on uniform grids with $N_x \in \{100, 200, 400, 800\}$. The trajectories for different resolutions overlap, indicating solution stability.
\textbf{(b) TD mass consistency:} Corresponding dynamics for the TD population $\bar W(t)$, exhibiting similar robustness to spatial discretization.}
    \label{fig:grid_convergence}
\end{figure}

\noindent
The rates are consistent with the expected first-order convergence of the upwind scheme
in the integrated $L^1$ sense. Higher-order accuracy could be attained with slope-limited MUSCL
reconstruction, but is not pursued here since the primary analytical conclusions rest on the
exact ODE closure.

\subsection{Test~2: ODE time-step and method independence}
\label{subsubsec:timestep}
We integrate the exact closure \eqref{eq:reduced_balance} with
$\Delta t_{\max}\in\{0.1,0.05,0.025,0.0125\}$ and compare to the finest-step trajectory at $T=10$.
Trajectory differences are $\mathcal O(10^{-11}\text{--}10^{-12})$ throughout, confirming that
temporal truncation error is negligible relative to other error sources.
Cross-method comparison (RK45 vs.\ DOP853 vs.\ Radau~IIA) yields agreement to
$\mathcal O(10^{-10}\text{--}10^{-11})$, providing independent confirmation.

\subsection{Test~3a: Balance-law residual}
\label{subsubsec:balance_residual}
Under uniform death $\delta(x)\equiv\delta$, Theorem~\ref{thm:balance-laws} and
Corollary~\ref{cor:uniform-death-closure} predict exact closure of the total-mass dynamics.
To ensure we are verifying the closure rather than measuring domain-truncation artefacts
(cf.\ Section~\ref{subsubsec:pde_vs_ode}), we use $x_{\max}=20$ with $N_x=1600$
($\Delta x\approx 0.0125$).

We define the computed balance residuals for the stem compartment as
\begin{equation}\label{eq:num_residual_P}
R_P(t) \coloneqq \frac{d\bar P_h}{dt}\bigg|_{\mathrm{num}}
- \Bigl[\bigl(p_1(\bar W_h)-p_2(\bar W_h)\bigr)\lambda_P(\bar W_h)\,\bar P_h
+ \lambda_R(\bar P_h)\,\bar W_h\Bigr],
\end{equation}
i.e.\ the numerical time derivative minus the exact right-hand side of the
$\bar P$-balance law \eqref{eq:balance_laws_general}.
The analogous residual $R_W(t)$ is defined using the $\bar W$-equation.
As an independent cross-check, we compute the total-mass residual
$R_M(t)$ both as $R_P(t)+R_W(t)$ and directly from the total-mass identity
$\dot M=\lambda_P(\bar W)\bar P-\delta\bar W$; agreement between these two forms
(to machine epsilon, $\approx 4\times 10^{-16}$)
provides an additional verification point.

At $T=5$, we observe $\|R_P\|_\infty,\|R_W\|_\infty,\|R_M\|_\infty$ at the
$\mathcal O(10^{-3})$ level (Figure~\ref{fig:balance_residual}).
A domain-sensitivity check confirms that these residuals are independent of $x_{\max}$
(identical at $x_{\max}\in\{10,15,20\}$ for $T=5$): at this time horizon the density support
has not reached $x_{\max}=10$, so truncation does not contaminate the residual.
The magnitudes of residuals decrease under grid refinement, confirming that the balance laws hold to the accuracy of
the spatial discretization.

\begin{figure}[htbp]
    \centering
    \includegraphics[width=\linewidth]{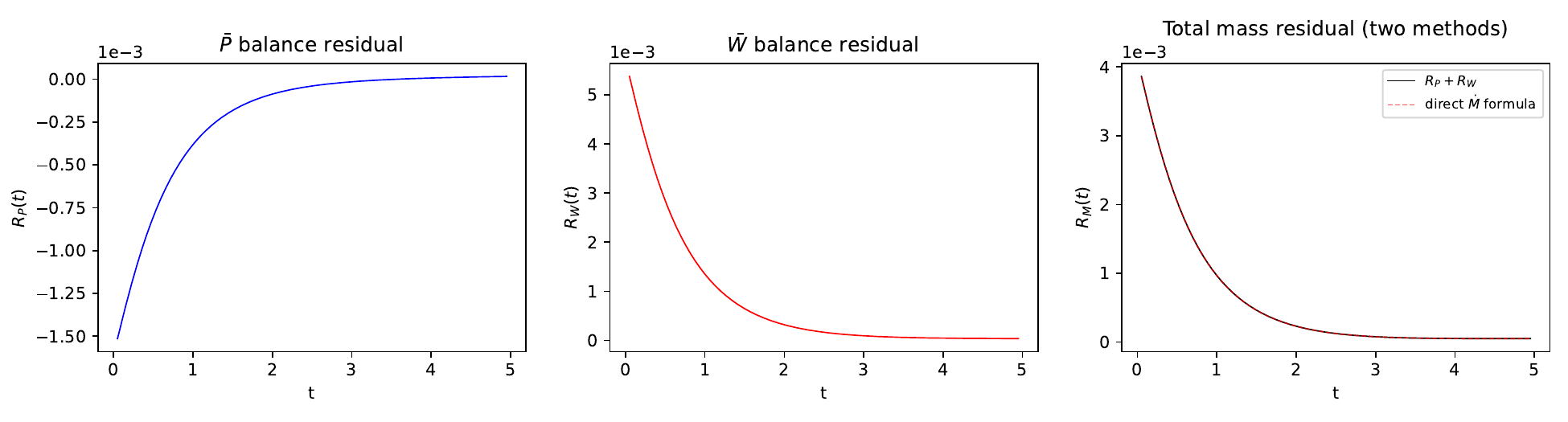}
    \caption{\textbf{Temporal evolution of the stem, TD, and total mass residuals}. The computed residuals for stem \textbf{(a)} and TD \textbf{(b)} masses remain bounded at the $\mathcal{O}(10^{-3})$ level, reflecting the expected first-order accuracy of the upwind finite-volume discretization.
\textbf{(c) Total mass residual consistency:} Consistency of $R_M(t)$ obtained from both the $R_P(t)+R_W(t)$ and the total-mass identity $\dot M = \lambda_P(\bar W)\bar P - \delta \bar W$.}
    \label{fig:balance_residual}
\end{figure}

\subsection{Test~3b: Identifiability check for uniform death via effective mortality}
\label{subsubsec:identifiability_delta}

When $\delta(x)$ is non-constant, the mortality flux in \eqref{eq:balance_laws_general} cannot be written as $\delta\,\bar W(t)$ and the two-dimensional closure \eqref{eq:reduced_balance} fails in general (Section~\ref{subsec:nonconstant_delta_breaks}, Fig.~\ref{fig:closure_condition}). A diagnostic quantity that makes this failure explicit is the \emph{effective mortality}
\begin{equation}\label{eq:delta_eff}
\delta_{\mathrm{eff}}(t):=\frac{\int_0^\infty \delta(x)W(t,x)\,dx}{\bar W(t)}\qquad(\bar W(t)>0).
\end{equation}
Exact closure on totals would require $\delta_{\mathrm{eff}}(t)\equiv\delta$ to be constant in time. Hence, any time-variation of $\delta_{\mathrm{eff}}(t)$ provides a direct, quantitative certificate that the uniform-death reduction is not valid for the simulated (or measured) regime.
In our uniform-death reference setting, \eqref{eq:delta_eff} is constant to machine precision; under damage-dependent $\delta(x)$ (not shown), $\delta_{\mathrm{eff}}(t)$ varies in time and the closure residuals increase accordingly, consistent with the mechanism in Fig.~\ref{fig:closure_condition}.

If $\delta(x)$ is weakly state-dependent, then $\delta_{\mathrm{eff}}(t)$ may remain approximately constant, yielding an approximate low-dimensional closure. Quantifying this approximation (e.g., via moment augmentation) is a natural extension, but is not required for the exact reduction results presented here.

We now present numerical solutions that independently corroborate each main theorem,
providing a figure-by-figure correspondence between analytical predictions and computed behaviour.

\subsection{PDE totals vs.\ exact ODE closure (Corollary~\ref{cor:uniform-death-closure})}
\label{subsubsec:pde_vs_ode}
As a direct check that the PDE total masses satisfy the closed ODE, we integrate both the full
PDE and the two-dimensional ODE from matched initial conditions over $T=15$.

Because the PDE is posed on a truncated domain $[0,x_{\max}]$ with zero padding beyond
$x_{\max}$ in the birth-operator interpolation, mass whose damage coordinate exceeds $x_{\max}$
is negligible.
To isolate this artefact from the spatial discretization error, we fix
$\Delta x\approx 0.0125$ and vary $x_{\max}\in\{10,15,20\}$
(scaling $N_x$ proportionally).
The PDE-vs-ODE discrepancy drops from $\mathcal O(10^{-1})$ at $x_{\max}=10$ to $\mathcal O(10^{-3})$ at
$x_{\max}=15$ and $\mathcal O(10^{-4})$ at $x_{\max}=20$, confirming that the dominant error source at
$x_{\max}=10$ is domain truncation rather than spatial discretization.

With $x_{\max}=20$ (sufficiently large that truncation effects are negligible),
the PDE-vs-ODE discrepancy in $\|\bar P_{\mathrm{PDE}}-\bar P_{\mathrm{ODE}}\|_{L^\infty[0,T]}$
decreases under grid refinement at the expected first-order rate
(Table~\ref{tab:pde_ode_conv}), directly corroborating that the closure
is exact and the remaining error is entirely due to the $\mathcal{O}(\Delta x)$ upwind scheme. Figure~\ref{fig:pde_vs_ode} shows the tight consistency between the totals in the PDE and the ODE models.
The positivity-fix mass defect is zero in all runs.

\begin{table}[h]
\centering
\caption{PDE-vs-ODE closure error under grid refinement ($x_{\max}=20$, $T=15$).
All entries are programmatically emitted.}
\label{tab:pde_ode_conv}
\begin{tabular}{ccccc}
\hline
$N_x$ & $\Delta x$ &
$\|\Delta\bar P\|_\infty$ & $\|\Delta\bar W\|_\infty$ & rate ($\bar P$) \\
\hline
 400 & 0.050 & $1.72\times 10^{-3}$ & $5.86\times 10^{-3}$ & --- \\
 800 & 0.025 & $8.50\times 10^{-4}$ & $2.89\times 10^{-3}$ & 1.02 \\
1600 & 0.0125 & $4.23\times 10^{-4}$ & $1.44\times 10^{-3}$ & 1.01 \\
3200 & 0.00625 & $2.11\times 10^{-4}$ & $7.16\times 10^{-4}$ & 1.00 \\
\hline
\end{tabular}
\end{table}

\begin{figure}
    \centering
    \includegraphics[width=\linewidth]{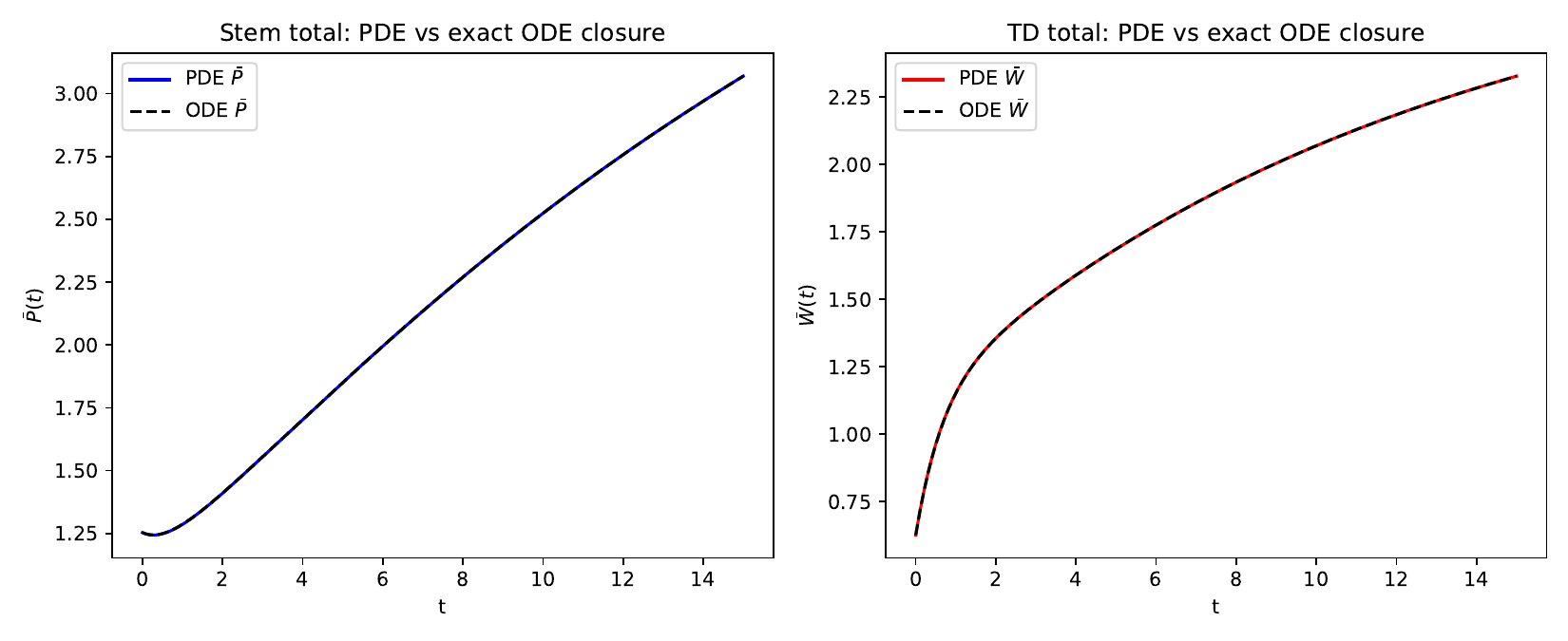}
    \caption{\textbf{Temporal evolution of the stem and TD total masses}.
    \textbf{(a) Stem total consistency:} The consistency of stem totals obtained from integrating the PDE solution and from solving the stem-total ODE. The overlap between the solid line (PDE) and the dashed line (ODE) validates the exact ODE closure. 
    \textbf{(b) TD total consistency:} The consistency of TD totals obtained from the same way as the stem totals.}
    \label{fig:pde_vs_ode}
\end{figure}

\subsection{Phase portrait with nullclines
(Theorems~\ref{thm:extinction_growth} and \ref{thm:global_convergence})}
\label{subsubsec:phase_portrait}
Figure~\ref{fig:phase_portrait} shows the $(\bar P,\bar W)$ phase plane with:
(i) the $\dot{\bar P}=0$ nullcline and the $\dot{\bar W}=0$ nullcline,
computed by solving the respective implicit equations via bisection;
(ii) a normalized vector field (arrows colored by speed);
(iii) eight trajectories launched from diverse initial conditions in $(0,\infty)^2$, integrated
with RK45 to $T=30$.
All trajectories converge monotonically to the unique interior equilibrium
$(\bar P^*,\bar W^*)$, with no oscillations or spiraling,
consistent with the strictly negative divergence established in Theorem~\ref{thm:no_cycles}
and the global convergence result of Theorem~\ref{thm:global_convergence}.
The origin $(0,0)$ is a saddle, confirming the growth regime:
for the reference parameters, $\Delta p=\hat p_1-\hat p_2=-0.20$ while
$\Delta p_{\mathrm{crit}}=-\hat\lambda_R/\delta=-0.60$, so
$\Delta p>\Delta p_{\mathrm{crit}}$
(Theorem~\ref{thm:extinction_growth}).

\begin{figure}[htbp]
    \centering
    \includegraphics[width=0.7\linewidth]{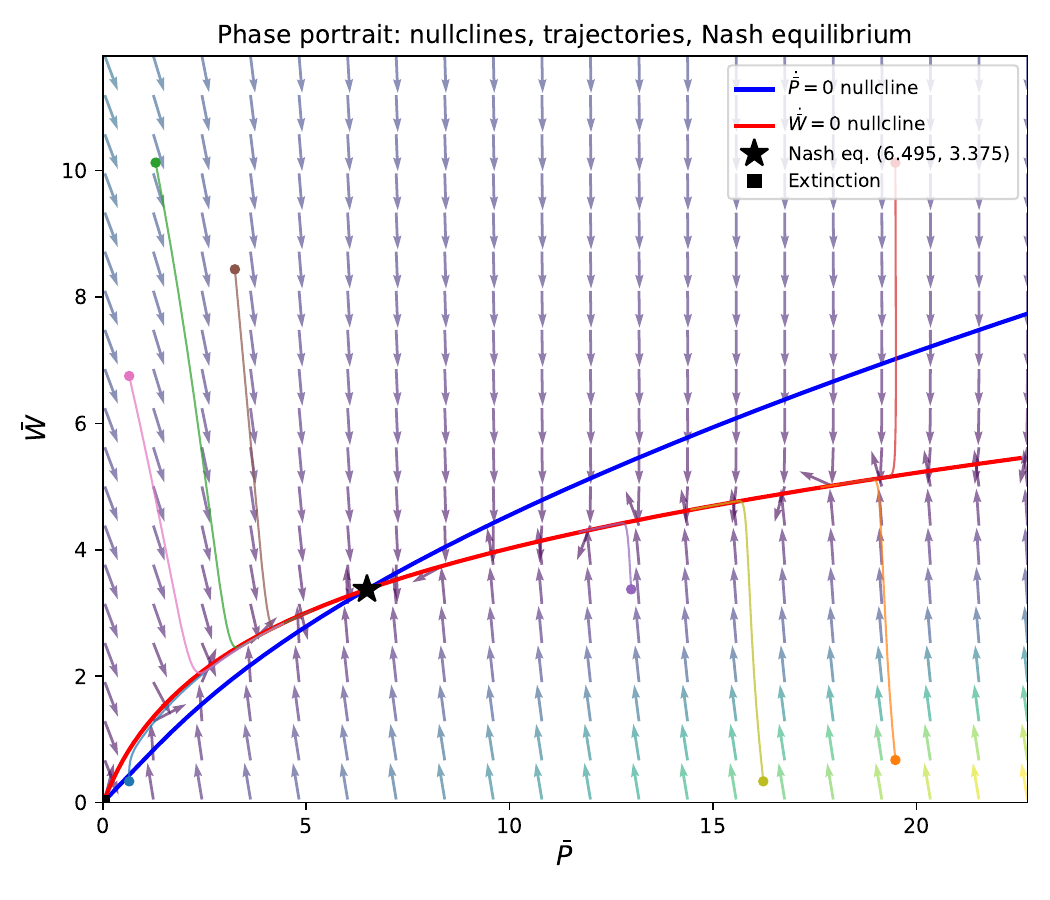}
    \caption{\textbf{Phase portrait of \eqref{eq:reduced_balance_recall}}. 
    The isoclines where the net growth rates vanish ($\dot{\bar{P}}=0$, blue) and ($\dot{\bar{W}}=0$, orange) divide the phase space into regions of monotonic increase or decrease.
    Their intersection defines the unique positive equilibrium $(\bar{P}^*, \bar{W}^*)$, representing the homeostatic tissue size.
    Sample trajectories (gray curves) initiated from various total mass configurations all converge to the equilibrium, numerically confirming the global stability guaranteed by the Bendixson–Dulac criterion.}
    \label{fig:phase_portrait}
\end{figure}

\subsection{Replicator field and Nash frequency
(Theorem~\ref{thm:nash_equiv})}
\label{subsubsec:replicator_evidence}
Figure~\ref{fig:replicator} plots
$\dot s = s(1-s)(F_P-F_W)$ as a function of the stem frequency $s=\bar P/M$
for several fixed values of the total mass $M\in\{0.3M^*,0.6M^*,M^*,1.5M^*,2.5M^*\}$,
where $M^*=\bar P^*+\bar W^*$.
In every case the replicator field has a unique zero-crossing at the Nash/homeostasis frequency
$s^*$ predicted by Theorem~\ref{thm:nash_equiv} and the equilibrium laws
\eqref{eq:equalization_law}--\eqref{eq:ratio_law}.
For $s<s^*$ the field is positive ($\dot s>0$: stem fraction increasing),
and for $s>s^*$ it is negative ($\dot s<0$: stem fraction decreasing),
confirming that $s^*$ is a stable rest point of the replicator equation \eqref{eq:replicator_form}
and hence an ESS of the induced two-phenotype game.

\begin{figure}[htbp]
    \centering
    \includegraphics[width=0.75\linewidth]{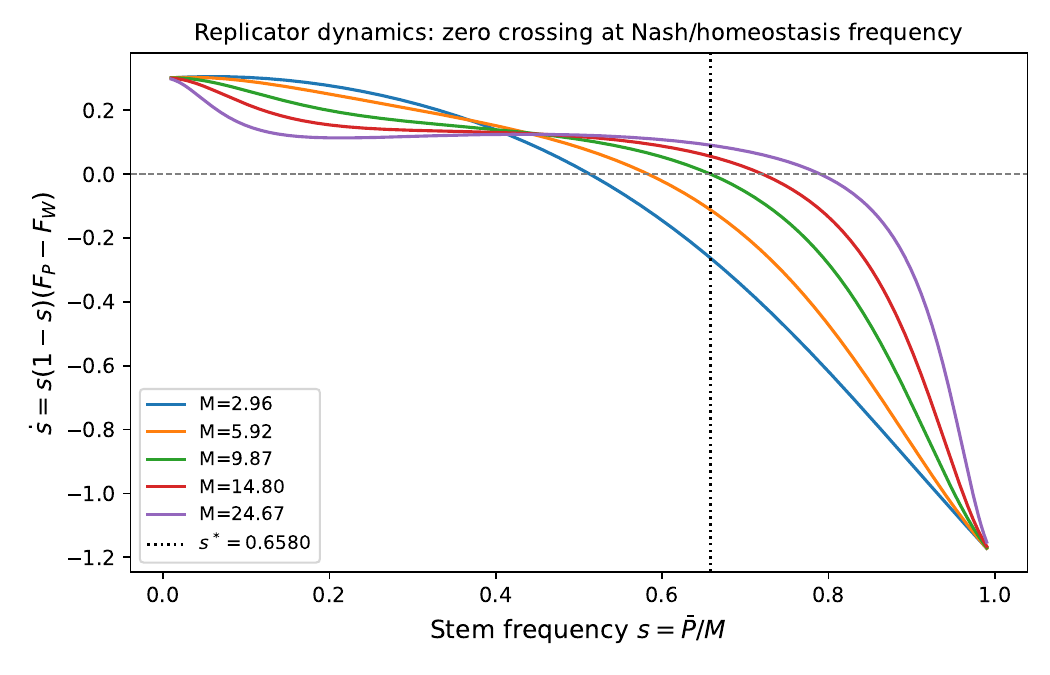}
    \caption{\textbf{Stem frequency $\dot s = s(1-s)(F_P-F_W)$}.
    The rate of change of stem frequency $\dot s$ is plotted against $s$ for fixed total mass values $M\in\{0.3M^*,0.6M^*,M^*,1.5M^*,2.5M^*\}$. The dynamics are governed by the replicator equation $\dot s = s(1-s)(F_P - F_W)$.
    All curves intersect $\dot s = 0$ at a unique interior frequency $s^*$, corresponding to the unique Nash equilibrium where the per-capita payoffs equalize ($F_P = F_W$).
    The negative slope at $s^*$ indicates that the equilibrium is evolutionarily stable (ESS), driving the population composition toward the homeostatic ratio regardless of the initial fraction.}
    \label{fig:replicator}
\end{figure}

\subsection{Bifurcation scan in \texorpdfstring{$\Delta p$}{Delta p} and \texorpdfstring{$\delta$}{delta}
(Theorem~\ref{thm:extinction_growth})}
\label{subsubsec:bifurcation}
Figure~\ref{fig:bifurcation} presents one-parameter continuation scans of the interior
equilibrium as $\Delta p:=\hat p_1-\hat p_2$ and $\delta$ are varied independently, with all
other parameters held at their reference values (Table~\ref{tab:params}). In panel (a),
$\Delta p$ is implemented through the admissible affine parameterization
$\hat p_1=(0.8+\Delta p)/3$ and $\hat p_2=(0.8-2\Delta p)/3$, so that
$\hat p_1-\hat p_2=\Delta p$ while the critical value
$\Delta p_{\mathrm{crit}}=-\hat\lambda_R/\delta$ remains in the interior of the scan window.

\begin{figure}[htbp]
    \centering
    \includegraphics[width=\linewidth]{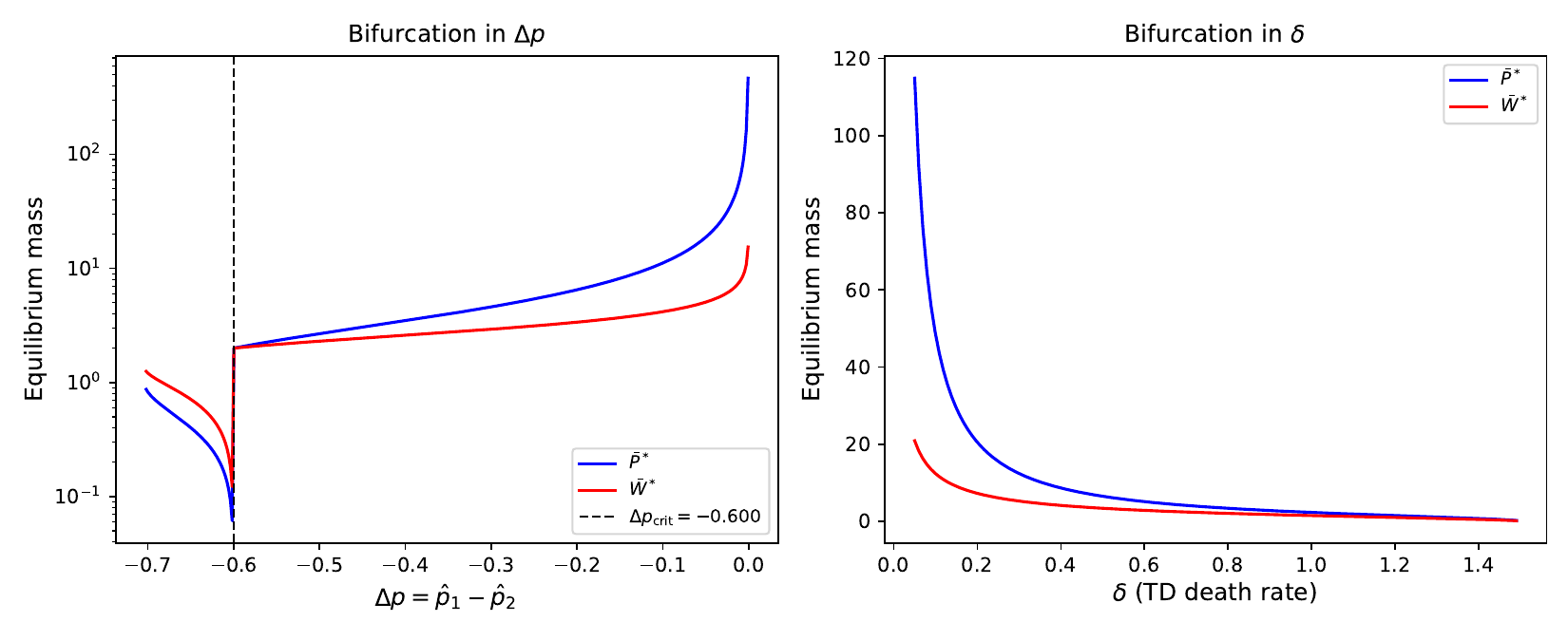}
    \caption{\textbf{One-parameter continuation scans}.
    	\textbf{(a) Critical threshold:} The equilibrium stem mass $\bar P^*$ (blue) and TD mass $\bar W^*$ (orange) are plotted against the renewal bias $\Delta p = \hat p_1 - \hat p_2$ using the affine parameterization $\hat p_1=(0.8+\Delta p)/3$, $\hat p_2=(0.8-2\Delta p)/3$. Both curves are displayed on a semilogarithmic scale. Gaps indicate parameter values for which either the probabilities are inadmissible or no positive equilibrium is detected. The dashed line marks $\Delta p_{\mathrm{crit}} = -\hat\lambda_R/\delta$.
    	\textbf{(b) Mortality sensitivity:} Dependence of the equilibrium masses on the TD death rate $\delta$. As mortality increases, the homeostatic tissue size contracts until extinction.}
    \label{fig:bifurcation}
\end{figure}

Theorem~\ref{thm:extinction_growth} establishes $\Delta p_{\mathrm{crit}}=-\hat\lambda_R/\delta$
as a local threshold at the origin: for $\Delta p<\Delta p_{\mathrm{crit}}$ the origin is
locally asymptotically stable, while for $\Delta p>\Delta p_{\mathrm{crit}}$ it is a saddle with an unstable manifold.
In the present parameter family that satisfies the hypotheses of
Theorems~\ref{thm:unique_equilibrium_strong}--\ref{thm:global_convergence} throughout
the scan range, the local instability of the origin coincides with the appearance of a
  unique positive equilibrium. Numerically, panel (a) resolves this branch over the admissible
  window $\Delta p\in[-0.8,0.4]$ determined by the chosen affine parameterization of
  $(\hat p_1,\hat p_2)$. The branch is omitted, rather than plotted at zero, whenever the
  probability constraints fail or the one-dimensional root solver detects no interior equilibrium.
  Within the displayed admissible range, the threshold $\Delta p_{\mathrm{crit}}=-0.6$ lies in the
  interior of the scan and separates the extinction regime from the positive branch. Near
  $\Delta p=0^-$ the equilibrium masses grow rapidly, so the semilogarithmic presentation is used
  to display the large variation in scale without conflating nonexistence with vanishing mass.
  In the $\delta$-scan, increasing $\delta$ shifts the equilibrium toward smaller total mass,
  primarily by increasing the net loss term in the $\bar W$-equation.
  No secondary bifurcation or multi-stability is observed in this parameter family.

\subsection{Gain-scaling invariance (Proposition~\ref{prop:gain_scaling})}
\label{subsubsec:scaling}
Proposition~\ref{prop:gain_scaling} predicts that under uniform amplification
$k_i\mapsto A k_i$ ($i=1,\dots,4$), the equilibrium rescales as
$(\bar P^*,\bar W^*)\mapsto(\bar P^*/A,\bar W^*/A)$
while the ratio $\bar P^*/\bar W^*$ is invariant.
Table~\ref{tab:scaling} confirms this to machine precision
(ratio errors $\le 3\times 10^{-15}$ for all $A$ tested),
providing a stringent check that both the analytical scaling law
and the numerical root-finder are correct. Figure~\ref{fig:scaling} shows the scaling of equilibrium levels and invariance of ratios under uniform scaling. 

\begin{table}[h]
\centering
\caption{Gain-scaling test: equilibrium under $k_i\mapsto A k_i$.
Let $\bar P^*_{A=1},\bar W^*_{A=1}$ denote the equilibrium at $A=1$.
The ratio $\bar P^*/\bar W^*$ is invariant to machine precision.}
\label{tab:scaling}
\begin{tabular}{rrrcrr}
\hline
$A$ & \multicolumn{1}{c}{$\bar P^*$} & \multicolumn{1}{c}{$\bar W^*$} &
$\bar P^*_A\big/\bigl(\bar P^*_{A=1}/A\bigr)$ &
$\bar P^*/\bar W^*$ & ratio error \\
\hline
0.5 & 12.9890 & 6.7507 & 1.0000 & 1.9241 & $2.89\times10^{-15}$ \\
1.0 &  6.4945 & 3.3753 & 1.0000 & 1.9241 & $0.00\times10^{0}$ \\
2.0 &  3.2473 & 1.6877 & 1.0000 & 1.9241 & $2.89\times10^{-15}$ \\
4.0 &  1.6236 & 0.8438 & 1.0000 & 1.9241 & $2.89\times10^{-15}$ \\
8.0 &  0.8118 & 0.4219 & 1.0000 & 1.9241 & $2.89\times10^{-15}$ \\
\hline
\end{tabular}
\end{table}

\begin{figure}
    \centering
    \includegraphics[width=\linewidth]{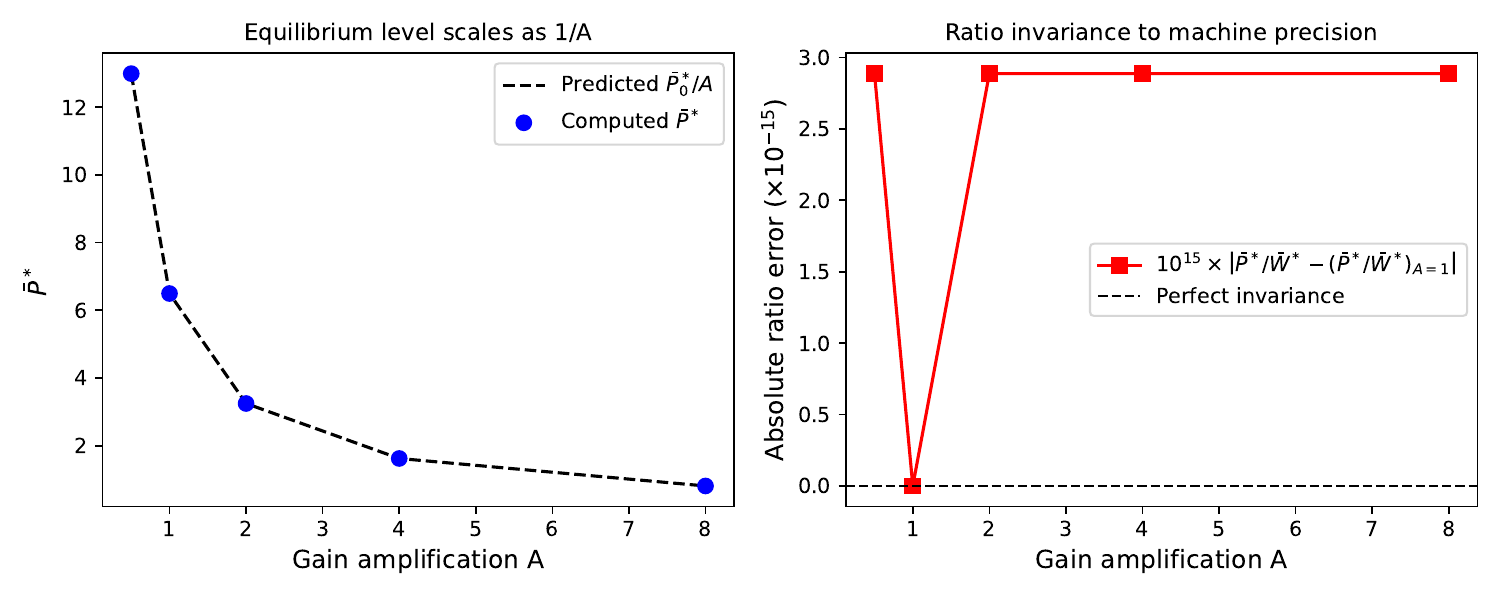}
    \caption{\textbf{Rescaling of the equilibrium and invariance of population ratio}.
    \textbf{(a) Gain rescaling:} The equilibrium stem mass $\bar P^*$ (blue dot) and predicted stem mass from Proposition~\ref{prop:gain_scaling} $\bar W^*$ (dashed curve) are plotted as a function of the feedback gain scaling factor $A$. As predicted by Proposition~\ref{prop:gain_scaling}, increasing the sensitivity of the feedback loops ($A \to 8.0$) scales the total tissue mass down exactly by $1/A$.
    \textbf{(b) Ratio invariance:} Rather than plotting the nearly constant raw ratio $\bar P^*/\bar W^*$, the panel reports the absolute deviation from the $A=1$ baseline ratio in units of $10^{-15}$. The deviations remain at machine precision across all tested scalings, confirming that the composition is invariant up to roundoff error.}
    \label{fig:scaling}
\end{figure}

\subsection{Certified Nash equilibrium computation
(Theorem~\ref{thm:equalization_ratio_laws})}
\label{subsubsec:nash_cert}
Algorithm~\ref{alg:nash} returns a unique interior equilibrium
$(\bar P^*,\bar W^*)$. The certification step evaluates four independent residual bounds:
\begin{center}
\begin{tabular}{lc}
\hline
\textbf{Certified residual} & \textbf{Bound} \\
\hline
Equalization law: $\bigl|(p_2(\bar W^*)-p_1(\bar W^*))\delta-\lambda_R(\bar P^*)\bigr|$ &
$< 10^{-16}$ \\
Ratio law: $\bigl|\bar P^*/\bar W^*-\delta/\lambda_P(\bar W^*)\bigr|$ &
$= 0$ \\
ODE RHS: $|\dot{\bar P}|$ at $(\bar P^*,\bar W^*)$ & $< 10^{-15}$ \\
ODE RHS: $|\dot{\bar W}|$ at $(\bar P^*,\bar W^*)$ & $< 10^{-15}$ \\
\hline
\end{tabular}
\end{center}
All residuals are at or below double-precision machine epsilon ($\approx 2.2\times 10^{-16}$),
certifying the equilibrium to the maximum attainable numerical accuracy.

\subsection{Negative divergence field (Theorem~\ref{thm:no_cycles})}
\label{subsubsec:divergence}
Figure~\ref{fig:divergence} displays the divergence
$\nabla\!\cdot\!\mathbf G(\bar P,\bar W)$ of the reduced vector field,
computed by centered finite differences on a $200\times 200$ grid over
$[0.01,\,3.5\bar P^*]\times[0.01,\,3.5\bar W^*]$.
The divergence is strictly negative at every sampled grid point;
the supremum over all $200\times 200$ evaluation points
satisfies $\sup\nabla\!\cdot\!\mathbf G<0$.
No zero-level contour appears.
This is consistent with the analytical result of
Theorem~\ref{thm:no_cycles}, which establishes strict negativity on all of $(0,\infty)^2$;
the numerical check provides independent corroboration on the sampled subdomain.

\begin{figure}
    \centering
    \includegraphics[width=0.75\linewidth]{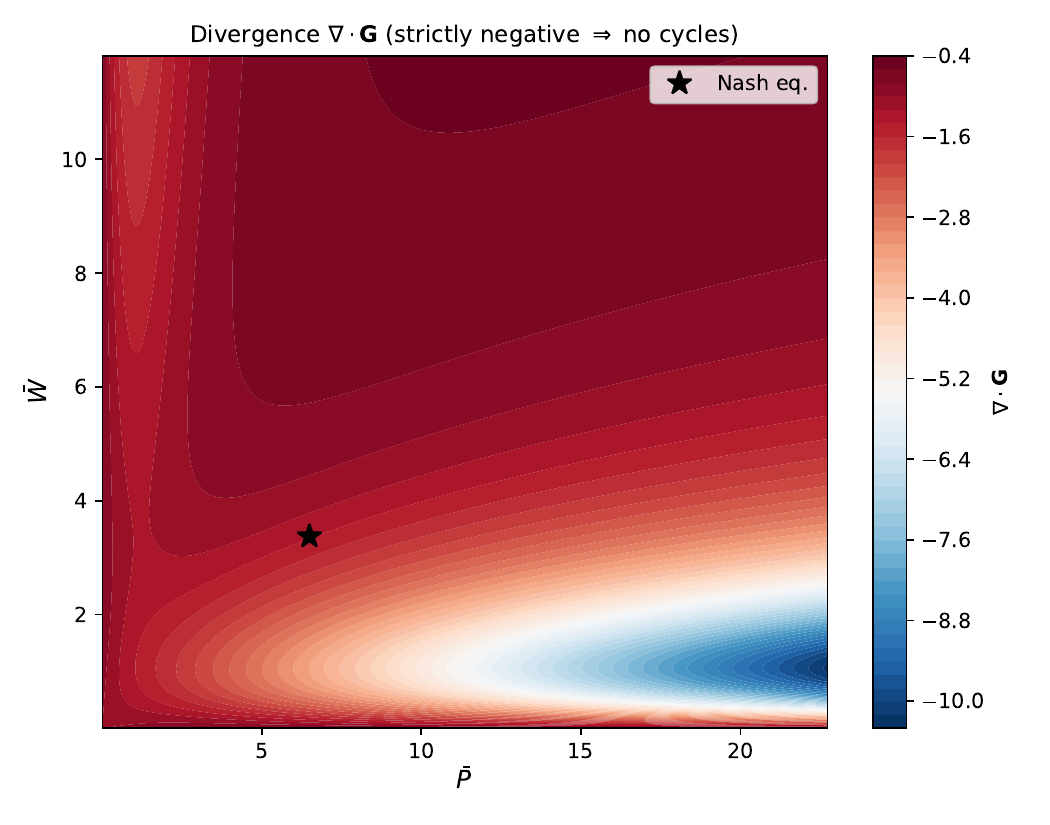}
    \caption{\textbf{Heatmap of the divergence $\nabla\cdot G$}.
    Heatmap of the vector field divergence $\nabla \cdot G = \partial_{\bar P} \dot{\bar P} + \partial_{\bar W} \dot{\bar W}$ across the phase space. The divergence is strictly negative everywhere, satisfying the Bendixson–Dulac criterion (with weighting function $\varphi \equiv 1$) and rigorously precluding the existence of periodic orbits. This global contraction of phase volume guarantees that all trajectories in the positive quadrant must converge to the unique homeostatic equilibrium.}
    \label{fig:divergence}
\end{figure}

\section{Case study: murine intestinal crypt homeostasis and regeneration}
\label{sec:case_study}

We illustrate the modeling pipeline on the murine small-intestinal crypt, a canonical
fast-renewing tissue in which stem-cell numbers, turnover times, and injury-induced
plasticity are unusually well quantified. In this system, Lgr5$^+$ crypt-base-columnar
(CBC) stem cells sustain epithelial renewal under homeostasis and can be replenished
after injury through progenitor-cell plasticity and dedifferentiation, making the crypt
a natural biological testbed for the equilibrium laws and regeneration dynamics derived
in Sections~\ref{sec:replicator_game}--\ref{sec:dynamics}
\citep{snippert_intestinal_2010,hagemann_quantitation_1970,van_der_flier_stem_2009,murata_ascl2-dependent_2020}.

\subsection{Biological context and aggregate calibration targets}
\label{subsec:crypt_biology}

Quantitative lineage-tracing experiments indicate that a murine small-intestinal crypt
contains approximately $14\pm2$ Lgr5$^{\mathrm{hi}}$ stem cells, and that these cells
divide roughly once per day under homeostatic conditions \citep{snippert_intestinal_2010}. At the
tissue level, the intestinal epithelium renews on a $\sim 4$--$5$ day timescale
\citep{van_der_flier_stem_2009}. As an aggregate calibration target for the present
two-compartment reduction, we take the total crypt cellularity to be approximately
292 cells, so that after assigning 14 cells to the Lgr5$^+$ stem compartment, the
remaining $\sim278$ cells define the effective TD compartment used below
\citep{hagemann_quantitation_1970}.

To probe regeneration, we use the radiation-injury paradigm. Lgr5-marked CBC cells are
highly radiosensitive and are strongly depleted after irradiation, while subsequent crypt
recovery requires restoration of the stem-cell pool \citep{richmond_enduring_2016}.
Recent lineage-tracing work further shows that this restoration is driven in large part
by ASCL2-dependent dedifferentiation of committed absorptive and secretory progenitors
back into the stem state \citep{murata_ascl2-dependent_2020}. These observations provide a natural
dynamical perturbation against which to compare the reduced model.

\subsection{Mapping the crypt onto the two-compartment reduction}
\label{subsec:crypt_mapping}

We identify the model stem compartment $P$ with the Lgr5$^+$ CBC population and the
effective TD compartment $W$ with all remaining crypt-resident TD cells.
This lumped $W$-compartment includes both transit-amplifying progenitors and more mature
epithelial lineages. Accordingly, the present case study should be read as an aggregate
coarse-graining of the crypt rather than as a one-to-one transcriptional cell-state model.

We set the effective loss rate in the $W$-compartment to
\[
\delta=\frac{1}{4}\;\mathrm{day}^{-1},
\]
corresponding to a representative 4-day epithelial turnover time. The homeostatic
aggregate calibration targets are therefore
\[
(\bar P^*,\bar W^*)=(14,278),
\qquad
s^*=\frac{\bar P^*}{\bar P^*+\bar W^*}=\frac{14}{292}=0.0479.
\]

\subsection{Calibration from the equilibrium laws}
\label{subsec:crypt_calibration}

A central advantage of the theory is that the homeostatic equilibrium is constrained
algebraically by the Ratio law and Equalization law of
Theorem~\ref{thm:equalization_ratio_laws}. Substituting the crypt targets
$(\bar P^*,\bar W^*)=(14,278)$ and $\delta=1/4$ into the Ratio law gives
\[
\lambda_P(\bar W^*)=\delta\frac{\bar W^*}{\bar P^*}
=\frac14\frac{278}{14}\approx 4.9643\;\mathrm{day}^{-1}.
\]
This value should be interpreted as an \emph{aggregate effective production rate} in the
two-compartment coarse-graining, not as the literal single-cell cycling rate of an
Lgr5$^+$ CBC cell. Biologically, this distinction is expected: the present reduction
lumps the transit-amplifying cascade into the effective TD compartment, so the
net stem-to-output flux is amplified relative to the $\sim1$ division/day CBC rate
\citep{snippert_intestinal_2010}.

We then calibrate Hill-type feedback maps to place the interior equilibrium exactly at
the chosen aggregate target. The calibrated parameter set returned by the script is
\begin{align*}
\hat p_1 &= 0.076603, &
\hat p_2 &= 0.222652, \\
k_{p_1} &= 0.004260, &
k_{p_2} &= 0.004260, \\
\hat\lambda_P &= 5.884502\;\mathrm{day}^{-1}, &
k_{\lambda_P} &= 0.001549, \\
\hat\lambda_R &= 0.039233\;\mathrm{day}^{-1}, &
k_{\lambda_R} &= 0.089821,
\end{align*}
with Hill exponents fixed at
\[
m_{p_1}=m_{p_2}=m_{\lambda_P}=m_{\lambda_R}=2,
\qquad
\delta=0.25\;\mathrm{day}^{-1}.
\]
For this calibrated model, the optimization objective is
$7.158\times10^{-28}$, the selected equilibrium is
\[
(\bar P^*,\bar W^*)=(14.0000,278.0000),
\]
the equalization residual is $-1.735\times10^{-18}$, and the ratio residual is
$6.939\times10^{-18}$. Thus, the equilibrium laws are satisfied up to numerical
roundoff, and the reduced system reproduces the prescribed aggregate homeostatic state
exactly at floating-point precision. 

\subsection{Homeostasis as payoff equalization}
\label{subsec:crypt_nash}

Under this calibration, the crypt homeostatic state is simultaneously:
(i) the interior equilibrium of the reduced two-compartment system,
(ii) the intersection of the ratio and Equalization laws, and
(iii) the interior rest point of the induced replicator equation for the stem frequency
\[
s(t)=\frac{\bar P(t)}{\bar P(t)+\bar W(t)}.
\]
Its equilibrium composition is
\[
s^*=0.0479,
\]
that is, a stem fraction of approximately $4.79\%$.

In the game-theoretic interpretation of Section~\ref{sec:replicator_game}, this means
that intestinal crypt homeostasis corresponds to payoff equalization between the stem
and TD phenotypes:
\[
F_P(\bar P^*,\bar W^*)=F_W(\bar P^*,\bar W^*)=0.
\]
Thus, the stem phenotype is not maintained because it is intrinsically dominant, but
because feedback-regulated lineage dynamics drive the system to a composition at which
neither phenotype has a per-capita growth advantage. In this sense, crypt homeostasis
is an induced Nash equilibrium of the reduced population game.
Figure~\ref{fig:crypt_main} shows murine intestinal crypt homeostasis and regeneration in the calibrated two-compartment reduction.

\begin{figure}[htbp]
    \centering
    \includegraphics[width=0.95\textwidth]{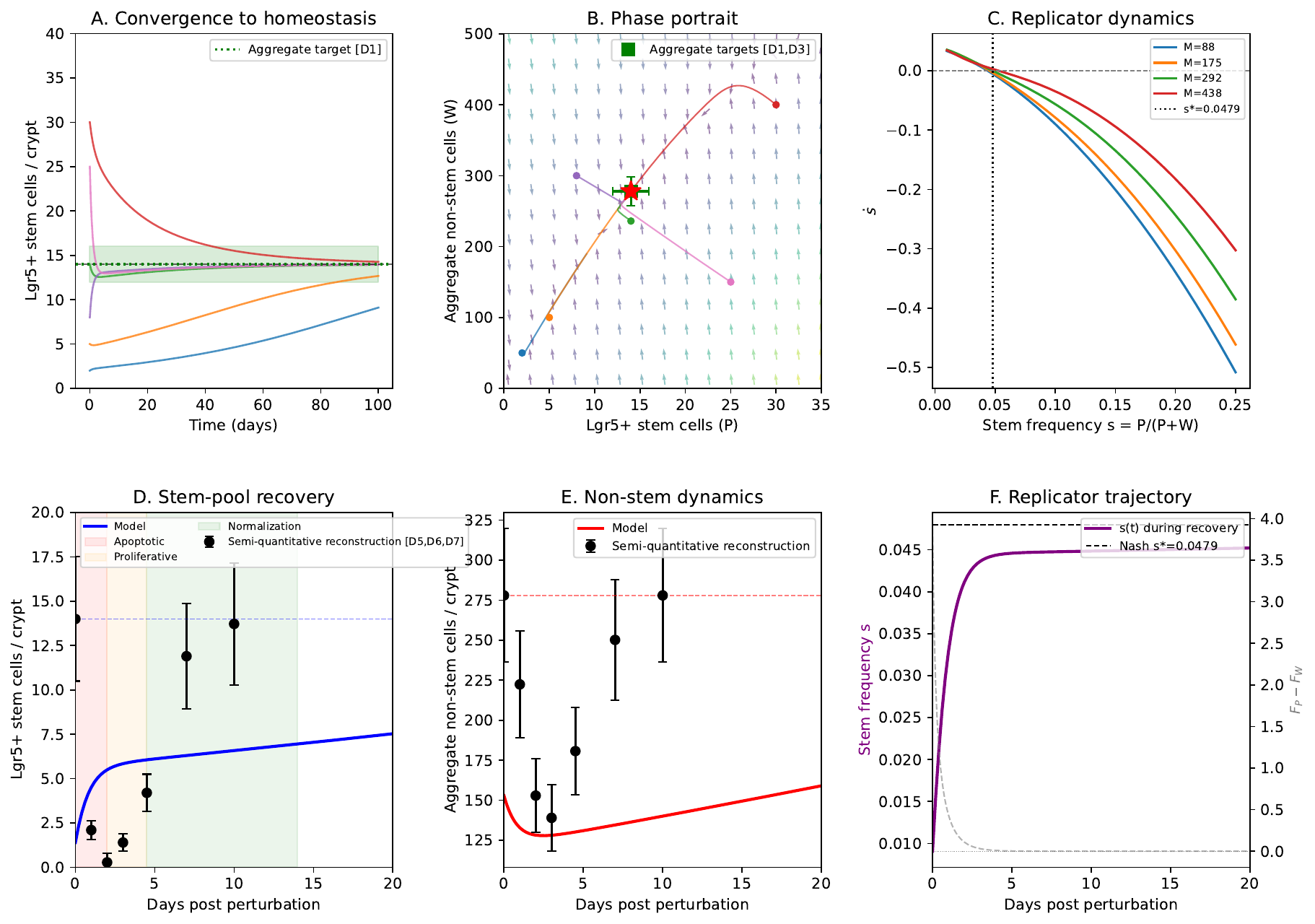}
    \caption{\textbf{Murine intestinal crypt homeostasis and regeneration in the calibrated two-compartment reduction.}
    \textbf{(A) Convergence to the homeostatic aggregate state.} Trajectories from diverse initial conditions converge to the calibrated equilibrium
    $(\bar P^*,\bar W^*)=(14,278)$, corresponding to a stem fraction
    $s^*=\bar P^*/(\bar P^*+\bar W^*)=0.0479$.
    \textbf{(B) Phase portrait and literature-based homeostatic target.} The nullclines of the reduced system intersect at the unique positive equilibrium. The highlighted target point corresponds to the aggregate homeostatic crypt state used for calibration, with approximately 14 Lgr5$^+$ stem cells and $\sim 292$ total crypt cells.
    \textbf{(C) Replicator dynamics and payoff equalization.} The induced frequency dynamics
    $\dot s=s(1-s)(F_P-F_W)$ admit a unique interior rest point at $s^*=0.0479$, showing that homeostasis corresponds to payoff equalization between stem and TD phenotypes.
    \textbf{(D) Stem-pool recovery after irradiation-like depletion.} Following acute depletion of the stem compartment, the model trajectory returns toward the homeostatic state. Black points denote a semi-quantitative reconstruction of the published post-irradiation response, included as a biological reference rather than as a single-cohort fit.
    \textbf{(E) TD compartment dynamics during regeneration.} The effective TD pool initially declines after injury and subsequently recovers as the system returns to the homeostatic composition.
    \textbf{(F) Recovery of stem frequency.} The stem fraction $s(t)$ returns toward the Nash/homeostatic frequency $s^*=0.0479$ after injury-induced depletion, illustrating that regeneration restores not only total abundance but also the homeostatic phenotypic composition. The biological context for the calibration is the murine small-intestinal crypt, in which Lgr5$^+$ CBC stem cells divide approximately once per day and the epithelium renews on a 4--5 day timescale. \label{fig:crypt_main}}
\end{figure}

\subsection{Post-injury regeneration and the role of dedifferentiation}
\label{subsec:crypt_regeneration}

To connect the equilibrium theory with regeneration, we consider an irradiation-like
perturbation in which the Lgr5$^+$ compartment is sharply depleted from its homeostatic
level, mimicking the experimentally observed radiosensitivity of CBC stem cells
\citep{richmond_enduring_2016}. In the numerical case study, the subsequent recovery
trajectory is compared against a semi-quantitative reconstruction of the
published post-irradiation response rather than a single harmonized cohort.

The key qualitative prediction is that recovery of the stem pool is driven initially by
the return flux $\lambda_R(\bar P)\bar W$ from the TD compartment rather than by
net self-renewal of the depleted stem population. This is consistent with the mechanism
highlighted experimentally by Murata et al. \citep{murata_ascl2-dependent_2020}, who showed that regeneration of the
ablated ISC compartment is explained to a large extent by dedifferentiation of
absorptive and secretory progenitors and requires ASCL2 activity.
Within the reduced model, the same biology appears as a transient regime in which the
stem frequency falls below $s^*$ after injury, generating a positive payoff difference
$F_P-F_W$ and driving the system back toward the homeostatic composition.
Figure~\ref{fig:crypt_dediffer} shows dedifferentiation as the dominant regenerative mechanism in the calibrated crypt model.
\begin{figure}[htbp]
    \centering
    \includegraphics[width=\textwidth]{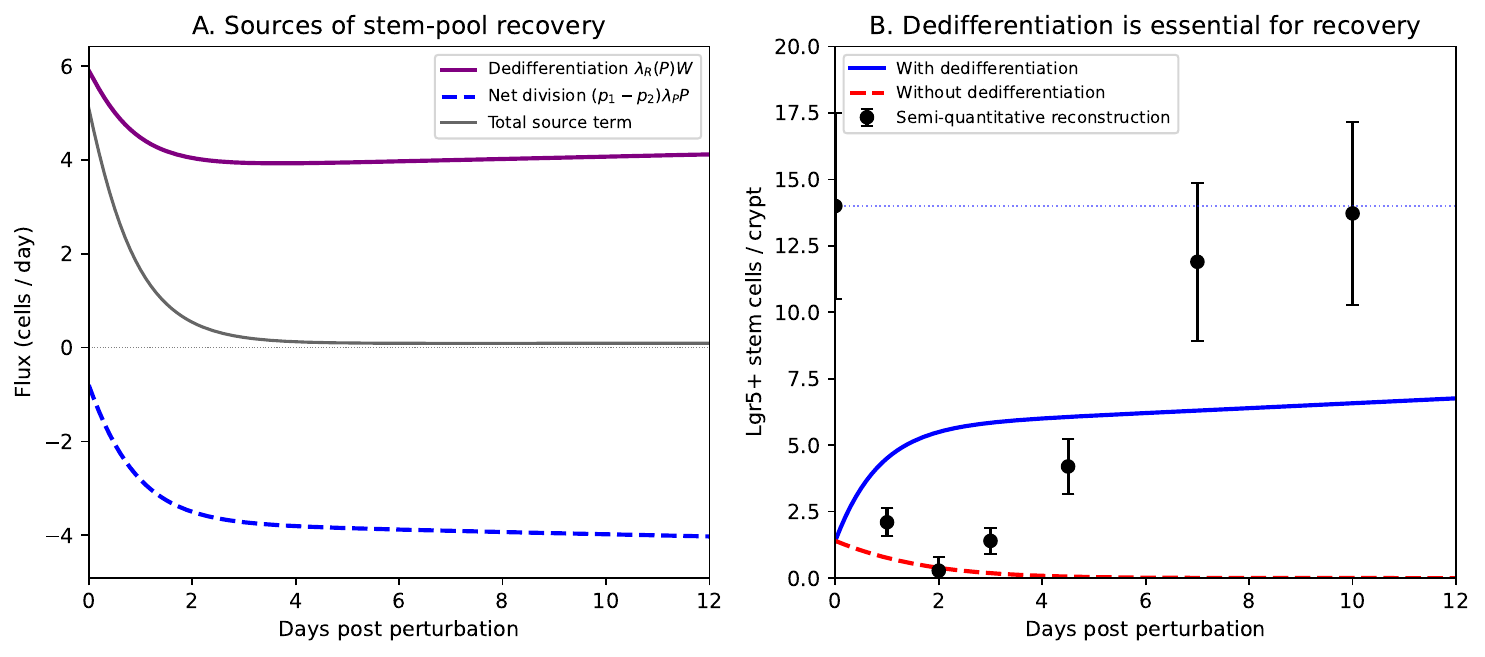}
    \caption{\textbf{Dedifferentiation is the dominant regenerative mechanism in the calibrated crypt model.}
    \textbf{(A) Source decomposition of stem recovery.} During the early recovery phase after irradiation-like depletion, the dominant positive contribution to stem restoration is the dedifferentiation flux
    $\lambda_R(\bar P)\bar W$, whereas the net division contribution
    $(p_1-p_2)\lambda_P\bar P$ is comparatively smaller and can remain negative during severe depletion. The total stem source term is their sum.
    \textbf{(B) Recovery with and without dedifferentiation.} Starting from the same post-injury depleted initial state, the full model (with dedifferentiation) restores the stem pool toward homeostasis, whereas the comparison model with $\lambda_R\equiv 0$ fails to recover the stem compartment. This mirrors the biological picture in which intestinal regeneration after stem-cell ablation depends strongly on progenitor-cell plasticity and ASCL2-dependent dedifferentiation.
    The comparison should be interpreted as an intervention on the regenerative dynamics from a common injury-depleted baseline, not as a claim that the no-dedifferentiation model shares the same pre-injury homeostatic state. \label{fig:crypt_dediffer}}
\end{figure}

This interpretation yields a clean mechanistic statement: following injury, the crypt
does not merely regrow; rather, it dynamically rebalances phenotypic payoffs until the
homeostatic equalization condition is restored. The replicator rest point therefore
organizes both steady maintenance and post-injury composition recovery.

\subsection{Biological predictions}
\label{subsec:crypt_predictions}

The calibrated case study yields three experimentally testable predictions.

First, the Ratio law predicts that perturbations to epithelial loss or shedding should
induce a compensatory shift in the equilibrium stem-to-TD ratio according to
\[
\frac{\bar P^*}{\bar W^*}=\frac{\delta}{\lambda_P(\bar W^*)}.
\]
Thus, increasing the effective loss rate $\delta$ should require either a larger stem
fraction or a larger effective production rate to preserve crypt homeostasis.

Second, the Equalization law predicts the dedifferentiation flux required to maintain
the homeostatic composition:
\[
\lambda_R(\bar P^*)=\delta\bigl(p_2(\bar W^*)-p_1(\bar W^*)\bigr).
\]
This provides a route for inferring otherwise hard-to-measure plasticity rates from
aggregate abundance measurements combined with proliferation and loss assays.

Third, the extinction--growth threshold of
Theorem~\ref{thm:extinction_growth} predicts a critical boundary in the balance between
renewal bias and loss rate. In the crypt setting, this implies that sufficiently severe
increases in epithelial attrition, or sufficiently strong suppression of regenerative
plasticity, should move the tissue toward a collapse regime in which a positive
homeostatic equilibrium no longer exists.

Figure~\ref{fig:crypt_pred} shows testable predictions of the calibrated intestinal-crypt case study.

\begin{figure}[htbp]
    \centering
    \includegraphics[width=\textwidth]{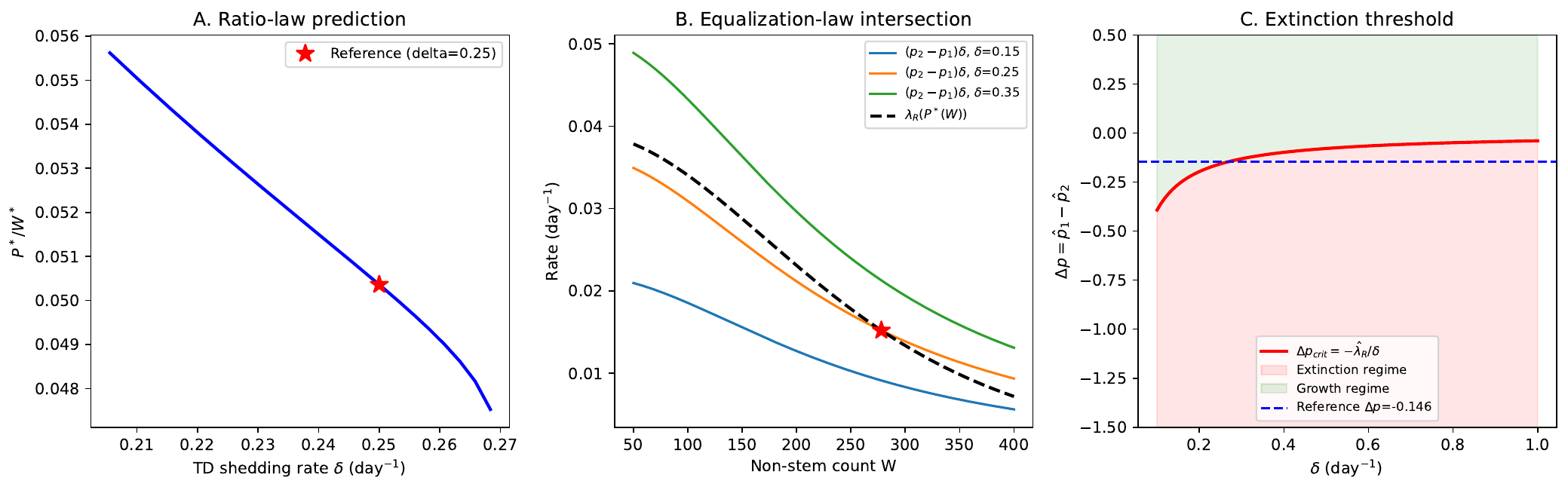}
    \caption{\textbf{Testable predictions of the calibrated intestinal-crypt case study.}
    \textbf{(A) Ratio-law prediction.} The equilibrium stem/TD ratio varies systematically with the effective TD loss rate $\delta$, in accordance with
    $\bar P^*/\bar W^*=\delta/\lambda_P(\bar W^*)$. In the present coarse-grained model, $\lambda_P(\bar W^*)$ is an aggregate effective production rate, so perturbations that alter epithelial loss or shedding are predicted to shift the equilibrium composition along this curve.
    \textbf{(B) Equalization-law geometry.} The homeostatic crypt state is determined by the intersection of the Ratio law and the Equalization law. The highlighted reference point marks the calibrated crypt equilibrium $(\bar P^*,\bar W^*)=(14,278)$.
    \textbf{(C) Extinction--growth threshold.} The reduced model predicts a threshold separating extinction and growth regimes, controlled by the balance between renewal bias and effective loss. Increasing epithelial attrition or suppressing regenerative plasticity moves the system toward the collapse boundary at which no positive homeostatic equilibrium persists.
    Together, these panels show how the aggregate laws derived from the mechanistic model yield directly falsifiable predictions for crypt composition, regenerative compensation, and loss of tissue integrity. \label{fig:crypt_pred}}
\end{figure}

\subsection{Limitations of the crypt coarse-graining}
\label{subsec:crypt_limitations}

The main biological limitation is the deliberate lumping of transit-amplifying cells
into the effective TD compartment. This is why the calibrated
$\lambda_P(\bar W^*)\approx4.9643\;\mathrm{day}^{-1}$ should not be compared directly to
the CBC single-cell cycling rate. A three-compartment stem/TA/TD extension would yield a
more literal mapping to crypt biology, but would also move the analysis away from the
exact two-dimensional closure established here.

A second limitation is that the irradiation-recovery trajectories used in the figures are
assembled as a semi-quantitative reconstruction from a literature spanning related but
not identical experimental contexts. Accordingly, the regeneration comparison should be
interpreted as a biologically grounded proof-of-principle case study rather than as a
formal statistical fit to one harmonized single-cohort dataset.

With these caveats, the murine intestinal crypt provides a compelling proof-of-principle
system: the aggregate homeostatic numbers are realistic, the injury-recovery mechanism is
biologically documented, and the induced game-theoretic laws generate concrete,
falsifiable predictions for how plasticity and feedback stabilize tissue maintenance.

\section{Discussion}
\label{sec:discussion}
\paragraph{Experimental interpretation of the homeostatic laws.}
At homeostasis, Eqs.~\eqref{eq:equalization_law}--\eqref{eq:ratio_law} can be read as two directly measurable balance statements. 
First, the Ratio Law, $\bar P^*/\bar W^*=\delta/\lambda_P(\bar W^*)$, predicts that the observed steady-state stem fraction is determined by a competition between TD clearance and the effective stem-cycle rate.
The TD clearance rate, $\delta$, can be measured using labeled-cohort survival or apoptosis assays.
The effective stem-cycle rate at the resident differentiated load, $\lambda_P(\bar W^*)$, can be estimated from EdU/BrdU incorporation, mitotic index, or time-lapse lineage tracking stratified by $\bar W^*$). 
Second, the Equalization Law, $\lambda_R(\bar P^*)=\delta\big(p_2(\bar W^*)-p_1(\bar W^*)\big)$, predicts the magnitude of the TD-to-stem return flux required to maintain a steady tissue composition. Experimentally, $\lambda_R(\bar P^*)$ can be inferred from fate-mapping or lineage tracing studies that quantify the rate at which marker-positive TD cells re-enter the stem compartment under steady-state conditions or following controlled micro-injury.
Together, these relations imply a concrete perturbation logic.
If plasticity is experimentally increased plasticity (e.g.\ injury cues that raise effective TD$\to$stem return), the steady ratio must shift along the algebraic curves in Fig.~\ref{fig:nash_geometry}. 
Conversely, if apoptosis or clearance $\delta$ is increased, the laws predict both the direction and magnitude of the compensatory shift in compartment sizes required to preserve equilibrium. 
The extinction--growth threshold (Theorem~\ref{thm:extinction_growth}) provides an orthogonal assay: depletion or repopulation experiments should display a sharp recovery vs.\ collapse boundary as renewal bias is tuned.

\subsection{Limitation and Extension}
\paragraph{On the assumption of uniform mortality}
Our derivation of the exact evolutionary game relies on the assumption of damage-independent mortality ($\delta(x) \equiv \delta$). 
Biologically, this regime corresponds to tissues where cell removal is driven primarily by extrinsic stochastic factors.
Examples include mechanical shedding in the intestinal epithelium \citep{blander_death_2016,williams_epithelial_2015}, random differentiation-linked loss, or environmental fluctuations, rather than the gradual accumulation of intracellular toxicity. 
In the context of ``neutral drift'' models of stem cell homeostasis \citep{xu_modeling_2018,traulsen_dynamics_2013}, clonal extinction is often modeled as a random process independent of fitness, aligning with our uniform rate hypothesis.

While strictly damage-dependent mortality ($\delta'(x) > 0$) would prevent the exact closure of the moment equations, it does not necessarily destroy the homeostatic stability. We argue that our reduced system captures the dominant order parameter (total tissue mass) dynamics. The analytical transparency gained by this simplification allows us to derive the explicit Equalization and Ratio Laws, offering a theoretical benchmark. Future work using moment-closure approximations could extend this framework to include weak damage-dependent mortality as a perturbation to this exact solution.

A natural extension is to incorporate intrinsic demographic fluctuations in division/death events. 
If spatial transport is introduced, one may seek a spectral analogue of the present scalar threshold laws. Recent mode-wise analyses of coupled mass-transport systems show that reciprocal feedback can be reduced to explicit Laplacian-mode instability criteria \citep{yu_mode-wise_2026}.
Recent chemical-Langevin work derived directly from continuous-time-Markov-chain event structure shows how to combine full state-dependent covariance, absorbing boundaries, and strong well-posedness in extinction-permitting population systems \citep{yu_full-covariance_2026}.
For stochastic population models, stochastic differential equation toolkits can provide well-posedness, moment control, and Lyapunov-type bounds (e.g., \citep{wang_analysis_2025-1}).
More generally, any stochastic extension of the present lineage model should be covariance-consistent at the event level rather than built from ad hoc diagonal noise. Recent mechanistic diffusion analyses emphasize that cross-covariance structure is dictated by channel stoichiometry and that open-domain and extinction-permitting absorbed formulations should be distinguished explicitly \citep{yu_beyond_2026}.
These tools could be combined with our exact closure to quantify how noise perturbs the ratio, equalization, and threshold predictions.

\begin{appendices}

\section{Technical lemmas and estimates}

\subsection{Notation and auxiliary bounds}
\label{app:notation_bounds}
For $T>0$ we work in
\[
\mathcal X_T:=C\!\big([0,T];L^1([0,\infty))\times L^1([0,\infty))\big),\qquad
\|(P,W)\|_{\mathcal X_T}:=\sup_{t\in[0,T]}\big(\|P(t,\cdot)\|_1+\|W(t,\cdot)\|_1\big).
\]
Throughout, $\|\cdot\|_1$ denotes the $L^1([0,\infty))$ norm, and we write
\[
\bar P(t)=\|P(t,\cdot)\|_1,\qquad \bar W(t)=\|W(t,\cdot)\|_1.
\]
Let $p_3(\cdot)=1-p_1(\cdot)-p_2(\cdot)$. Under Assumption~\ref{ass:basic}, the feedback maps are
bounded and Lipschitz on $[0,\infty)$. We set the global bounds
\begin{equation}\label{app:eq:bounds}
0\le p_i(\cdot)\le \hat p_i,\qquad 0\le p_3(\cdot)\le 1,\qquad
0\le \lambda_P(\cdot)\le \hat\lambda_P,\qquad 0\le \lambda_R(\cdot)\le \hat\lambda_R,
\end{equation}
and the global Lipschitz constants
\begin{equation}\label{app:eq:Lipschitz_constants}
L_{p_1}:=\sup_{x\ge 0}|p_1'(x)|,\quad
L_{p_2}:=\sup_{x\ge 0}|p_2'(x)|,\quad
L_{\lambda_P}:=\sup_{x\ge 0}|\lambda_P'(x)|,\quad
L_{\lambda_R}:=\sup_{x\ge 0}|\lambda_R'(x)|.
\end{equation}
Then $p_3$ is Lipschitz with constant $L_{p_3}:=L_{p_1}+L_{p_2}$.

\subsection{Kernel mass identities}
\label{app:kernel_mass}
A key structural ingredient in the balance-law reduction is that the nonlocal birth operators
preserve $L^1$-mass in each deterministic inheritance branch.

\begin{lemma}[Change-of-variables identity]\label{lem:cov_identity}
For any $\alpha\in(0,1)$ and any $f\in L^1_+([0,\infty))$,
\begin{equation}\label{app:eq:cov}
\int_0^\infty \frac{1}{\alpha}\, f\!\left(\frac{x}{\alpha}\right)\,dx
=\int_0^\infty f(y)\,dy.
\end{equation}
\end{lemma}

\begin{proof}
Set $y=x/\alpha$, so $dx=\alpha\,dy$ and the integral becomes
$\int_0^\infty f(y)\,dy$.
\end{proof}

\begin{lemma}[Birth-operator mass formulas]\label{lem:birth_mass}
Let $\mathcal B_P,\mathcal B_W$ be defined in \eqref{eq:birth_ops}. Then, for a.e.\ $t\ge 0$,
\begin{align}
\label{app:eq:BP_mass}
\int_0^\infty \mathcal B_P[P](t,x)\,dx
&=\big(2p_1(\bar W(t))+p_3(\bar W(t))\big)\lambda_P(\bar W(t))\,\bar P(t),\\
\label{app:eq:BW_mass}
\int_0^\infty \mathcal B_W[P](t,x)\,dx
&=\big(2p_2(\bar W(t))+p_3(\bar W(t))\big)\lambda_P(\bar W(t))\,\bar P(t),
\end{align}
and therefore
\begin{equation}\label{app:eq:BPBW_sum}
\int_0^\infty \big(\mathcal B_P[P]+\mathcal B_W[P]\big)(t,x)\,dx
=2\,\lambda_P(\bar W(t))\,\bar P(t).
\end{equation}
\end{lemma}

\begin{proof}
Apply Lemma~\ref{lem:cov_identity} term-by-term in \eqref{eq:birth_ops} and use
$p_1+p_2+p_3=1$.
\end{proof}

\subsection{Vanishing of transport boundary flux under homogeneous inflow}
\label{app:boundary_flux}
The following lemma formalizes why the boundary terms disappear when integrating the transport
equations over $x\ge 0$ under the homogeneous inflow condition.

\begin{lemma}[Boundary flux cancellation]\label{lem:boundary_flux}
Assume $v>0$ and $u\in C([0,T];L^1([0,\infty)))$ is a weak solution to
$\partial_t u+v\partial_x u=g$ on $(0,T)\times(0,\infty)$ with homogeneous inflow boundary
$u(t,0)=0$ (in the trace sense) and $g\in L^1((0,T);L^1)$.
Then, for a.e.\ $t\in(0,T)$,
\begin{equation}\label{app:eq:flux_zero}
\int_0^\infty v\,\partial_x u(t,x)\,dx=0
\end{equation}
in the distributional sense used to derive balance laws.
\end{lemma}

\begin{proof}[Proof sketch]
Test the weak formulation against $\psi_R(x)$ with $\psi_R\equiv 1$ on $[0,R]$,
$\mathrm{supp}\,\psi_R\subset[0,R+1]$, and $\|\psi_R'\|_\infty\le C$ independent of $R$.
Integration by parts yields the transport contribution
$v\int_0^\infty u(t,x)\psi_R'(x)\,dx$, which vanishes as $R\to\infty$ by dominated convergence and
$u(t,\cdot)\in L^1$. The inflow boundary condition ensures there is no boundary injection at $x=0$.
\end{proof}

\subsection{Mild formulation and Picard contraction}
\label{app:picard}
We record a clean version of the mild formulation and a contraction estimate sufficient for local
well-posedness.

Define characteristics
\[
X_P(s;t,x)=x-v_P(t-s),\qquad X_W(s;t,x)=x-v_W(t-s),\qquad 0\le s\le t,
\]
and integrating factors
\[
\Lambda_P(s;t):=\int_s^t \lambda_P(\bar W(\tau))\,d\tau,\qquad
\Lambda_W(s;t,x):=\int_s^t\big(\delta(X_W(\tau;t,x))+\lambda_R(\bar P(\tau))\big)\,d\tau.
\]
For $(P,W)\in\mathcal X_T$, define $\Psi(P,W)=(\widetilde P,\widetilde W)$ by
\begin{align}
\label{app:eq:mildP}
\widetilde P(t,x)
&=e^{-\Lambda_P(0;t)}\,P_0(x-v_P t)\,\mathbf 1_{x\ge v_P t}
+\int_0^t e^{-\Lambda_P(s;t)}\Big(\mathcal B_P[P]+\lambda_R(\bar P(s))W\Big)\!\big(s,X_P(s;t,x)\big)\mathbf 1_{x\ge v_P(t-s)}\,ds,\\
\label{app:eq:mildW}
\widetilde W(t,x)
&=e^{-\Lambda_W(0;t,x)}\,W_0(x-v_W t)\,\mathbf 1_{x\ge v_W t}
+\int_0^t e^{-\Lambda_W(s;t,x)}\mathcal B_W[P]\!\big(s,X_W(s;t,x)\big)\mathbf 1_{x\ge v_W(t-s)}\,ds.
\end{align}

\begin{lemma}[Local contraction estimate]\label{lem:contraction}
Let $M_0:=\|P_0\|_1+\|W_0\|_1$ and consider the closed ball
$\overline{\mathbb B}(0,2M_0)\subset\mathcal X_T$. There exists a constant $C=C(M_0)$, depending only
on $M_0$, the bounds \eqref{app:eq:bounds}, and the Lipschitz constants
\eqref{app:eq:Lipschitz_constants}, such that for all $(P_j,W_j)\in \overline{\mathbb B}(0,2M_0)$,
\begin{equation}\label{app:eq:contract_est}
\|\Psi(P_1,W_1)-\Psi(P_2,W_2)\|_{\mathcal X_T}
\le C\,(T+T^2)\,\|(P_1,W_1)-(P_2,W_2)\|_{\mathcal X_T}.
\end{equation}
In particular, for $T>0$ sufficiently small so that $C(T+T^2)<1$, $\Psi$ is a strict contraction on
$\overline{\mathbb B}(0,2M_0)$.
\end{lemma}

\begin{proof}[Proof skeleton]
The proof proceeds term-by-term in \eqref{app:eq:mildP}--\eqref{app:eq:mildW}:
(i) differences of exponentials are bounded by integrals of coefficient differences, using the mean
value theorem and Lipschitz continuity;
(ii) the nonlocal birth terms satisfy
\[
\Big\|\frac{1}{\alpha}P(\cdot,\cdot/\alpha)\Big\|_{L^1}=\|P\|_{L^1}
\]
by Lemma~\ref{lem:cov_identity}, and coefficient differences,
controlled by the Lipschitz property of $\lambda_P$ and the uniform bound $\|P\|_1 \le 2M_0$ on the ball, contribute at most $\mathcal{O}(T)$;
(iii) nested time integrals yield the $\mathcal{O}(T^2)$ term.
All bounds close on $\overline{\mathbb B}(0,2M_0)$ because $\bar P,\bar W$ are uniformly bounded by
$2M_0$ on this set.
\end{proof}

\subsection{A priori mass estimate (global continuation)}
\label{app:mass_est}
Finally, we record the $L^1$ mass bound used to extend local solutions globally.

\begin{lemma}[Mass bound]\label{lem:mass_bound}
Let $(P,W)$ be a nonnegative mild solution on $[0,T]$ under Assumption~\ref{ass:basic}.
Define $M(t):=\|P(t)\|_1+\|W(t)\|_1$. Then
\begin{equation}\label{app:eq:mass_bound}
M(t)\le M(0)\,e^{\hat\lambda_P t},\qquad t\in[0,T],
\end{equation}
where $\hat\lambda_P=\lambda_P(0)$.
\end{lemma}

\begin{proof}
Integrate the PDEs over $x\ge 0$ using Lemmas~\ref{lem:birth_mass} and~\ref{lem:boundary_flux}.
Nonnegativity and $\lambda_P(\bar W)\le\hat\lambda_P$ yield
$\dot M(t)\le \hat\lambda_P\,\bar P(t)\le \hat\lambda_P\,M(t)$, hence \eqref{app:eq:mass_bound}
by Gr\"onwall.
\end{proof}

\end{appendices}

\medskip\noindent\textbf{Acknowledgments:}  
The authors confirm that this research received no external funding.

\medskip\noindent\textbf{Author contributions:}  
Both authors contributed equally to this article and as Co-First Authors. Both two authors read and approved the final
manuscript

\medskip\noindent\textbf{Competing interests:}  
The authors declare that they have no conflict of interest.

\medskip\noindent\textbf{Data availability:}  
Data sharing is not applicable to this article as no datasets were generated or analyzed during the current study.

\bibliography{reference}

\end{document}